\documentclass[10pt]{amsart}
\usepackage[latin1]{inputenc}
\usepackage[T1]{fontenc}
\usepackage{amsmath,amscd}
\usepackage{graphicx}
\usepackage[all]{xy}
\usepackage{amsthm}
\usepackage{amssymb,url}
\usepackage[charter]{mathdesign}
\usepackage{amsfonts}
\usepackage{mathrsfs}
\usepackage{amsxtra}
\usepackage{bbm}
\usepackage{ae}
\usepackage[all]{xy}
\usepackage{color}

\newtheorem{thm}{Theorem}[section]
\newtheorem{thm*}{Theorem}
\newtheorem{lem}[thm]{Lemma}
\newtheorem{cor}[thm]{Corollary}
\newtheorem{cor*}[thm*]{Corollary}

\newtheorem{prop}[thm]{Proposition}
\newtheorem{prop*}[thm*]{Proposition}

\theoremstyle{remark}
\newtheorem{remark}[thm]{Remark}
\newtheorem{remark*}[thm*]{Remark}

\theoremstyle{definition}
\newtheorem{deef}[thm]{Definition}
\newtheorem{deef*}[thm*]{Definition}
\newtheorem*{notation}{Notation}

\newcommand{\R}{\mathbbm{R}}
\newcommand{\C}{\mathbbm{C}}

\newcommand{\Z}{\mathbbm{Z}}

\newcommand{\N}{\mathbbm{N}}

\newcommand{\id}{ \mathrm{id}}
\newcommand{\rd}{\mathrm{d}}

\newcommand{\Ko}{\mathbbm{K}}

\newcommand{\Bo}{\mathcal{B}}

\newcommand{\He}{\mathcal{H}}
\newcommand{\T}{\mathbbm{T}}

\newcommand{\Dom}{\mathrm{Dom}\,}

\newcommand{\End}{\mathrm{End}}

\newcommand{\Lip}{\textnormal{Lip}}
\renewcommand{\epsilon}{\varepsilon}

\newcommand{\im}{\mathrm{i} \mathrm{m} \,}

\newcommand{\supp}{\mathrm{s} \mathrm{u} \mathrm{p} \mathrm{p}\,}
\newcommand{\e}{\mathrm{e}}
\newcommand{\tra}{\mathrm{t}\mathrm{r}}

\title[Nonclassical spectral asymptotics and Dixmier traces]{Nonclassical spectral asymptotics and Dixmier traces: From circles to contact manifolds}
\author{Heiko Gimperlein, Magnus Goffeng}
\address{Heiko Gimperlein,\newline
\indent Maxwell Institute for Mathematical Sciences and \newline
\indent Department of Mathematics, Heriot-Watt University\newline
\indent Edinburgh EH14 4AS\newline
\indent United Kingdom\newline
\newline
\indent {\it and}\newline
\newline 
\indent Institute for Mathematics, University of Paderborn\newline
\indent Warburger Str.~100\newline
\indent 33098 Paderborn\newline
\indent Germany\newline
\newline
\indent Magnus Goffeng,\newline
\indent Department of Mathematical Sciences\newline 
\indent Chalmers University of Technology and \newline 
\indent University of Gothenburg\newline 
\indent SE-412 96 Gothenburg\newline 
\indent Sweden\newline}
\email{h.gimperlein@hw.ac.uk, goffeng@chalmers.se}

\begin{document}
\maketitle

\begin{abstract}

We consider the spectral behavior and noncommutative geometry of commutators $[P,f]$, where $P$ is an operator of order $0$ with geometric origin and $f$ a multiplication operator by a function. When $f$ is H\"{o}lder continuous, the spectral asymptotics is governed by singularities. We study precise spectral asymptotics through the computation of Dixmier traces; such computations have only been considered in less singular settings. Even though a Weyl law fails for these operators, and no pseudo-differential calculus is available, variations of Connes' residue trace theorem and related integral formulas continue to hold. On the circle, a large class of non-measurable Hankel operators is obtained from H\"older continuous functions $f$, displaying a wide range of nonclassical spectral asymptotics beyond the Weyl law. The results extend from Riemannian manifolds to contact manifolds and noncommutative tori.
\end{abstract}

\large
\section{Introduction}
\normalsize

Let $M$ be a closed Riemannian manifold and $P$ a classical pseudo-differential operator of order $0$. If $f \in C^\infty(M)$, the singular values $\mu_k$ of the commutator $[P,f]$ satisfy a Weyl law, as determined by the Poisson bracket $\frac{1}{i}\{\sigma(P), f\}$ with the principal symbol $\sigma(P)$: 
\begin{equation}
\label{weyllaw}
\mu_k([P,f])= \left(\frac{1}{n(2\pi)^n}\int_{S^*M} |\{\sigma(P),f\}|^n\rd x \rd \xi\right)^{1/n} \cdot k^{-1/n}+o(k^{-1/n}),\quad\mbox{as $k\to\infty$}\ ,
\end{equation}
where $n = \dim M$. Connes' residue trace theorem allows to interpret $[P,f]$ as a noncommutative differential form \cite{connesaction, c}. Dixmier traces relate to an averaged form of the Weyl law \eqref{weyllaw}. By computing them for operators related to $[P,f]$, we obtain geometric consequences of non-Weyl behaviour.   \\

In this article we find a rich spectral behavior and noncommutative geometry of operators of the form $[P,f]$, when $f$ is merely H\"{o}lder continuous.  In this case the spectral asymptotics is governed by the singularities of $f$. For $f$ from a generic set, suitable powers or products of operators of the form $[P,f]$ turn out to produce natural examples of non-measurable operators. We explore new phenomena due to the singularities, and extensions of results for smooth functions to the H\"{o}lder classes. Already for the circle $M = S^1$ and $P$ the Szeg\"{o} projection (the projection onto the closed linear span of the positive Fourier modes), we obtain an analogue of the residue trace theorem and explicit geometric integral formulas, in spite of the highly nonclassical behavior of the singular values. Hochschild and cyclic cocycles defined from the derivation $f\mapsto [P,f]$ and (singular) traces prove relevant to study the algebras of  H\"{o}lder functions.

These basic results extend from $M=S^1$ to Riemannian and contact manifolds $M$, as well as to noncommutative tori. The singularities studied in this article are in keeping with the structures which appear for a singular manifold; they are quite similar to a manifold with H\"{o}lder charts.

\subsection{Commutators and Dixmier traces in noncommutative geometry}

In recent years, Dixmier traces, and more generally singular traces, have been studied extensively both in the theory of operator ideals and as a notion of integral in noncommutative geometry. The book by Lord, Sukochev and Zanin \cite{sukolord} provides a reference for this  progress, see also \cite{kaloposu, pietsch, pietschtwo,lidskiitypesesuza,dixmieragain}. It is our aim to show how the ideas from these works apply to highly singular geometric settings, to nonsmooth commutators.

The view of commutators as noncommutative differential forms goes far back. It is based on the algebraic principle that a derivation generalizes differentiation. The approach we take is based on Connes' noncommutative geometry \cite[Chapter III]{c}. Together with (singular) traces on operators ideals, commutators and their products provide a natural framework for constructing cyclic and Hochschild cocycles -- noncommutative analogues of de Rham cycles and currents, respectively. Examples of their use with relations to this work include \cite{conneskaroubi,kaadcomparison} where pairings with relative and algebraic $K$-theory were considered. Their application to classical geometry, for example, leads to explicit analytic formulas for the mapping degree of nonsmooth mappings \cite{goffodd}. In the latter work, only nonsharp spectral properties of commutators are required; they can be deduced from a general theorem of Russo \cite{russo} for Schatten class properties of integral operators.

Beyond Riemannian geometry, Dixmier traces have also been studied in sub-Riemannian settings, for example relating Dixmier traces with certain integrals in complex analysis \cite{engroch, engzh,engzhtwo}. For general sub-Riemannian $H$-manifolds, Ponge \cite{pongeresidue} has extended Connes' residue trace theorem to the Beals-Greiner calculus of pseudodifferential operators. Some of these results are extended to the nonsmooth setting in this paper.
\\

Here we combine the above two directions. We use singular traces as a tool for new constructions in cyclic cohomology rather than their classical use as integrals. We find that Dixmier traces of nonsmooth commutators and the associated cyclic cocycles are geometrically relevant and computable. In noncommutative geometry we define exotic, nontrivial cyclic cocycles on the algebras $C^\alpha$ of H\"{o}lder functions, spanning an infinite-dimensional subspace of the cyclic cohomology without a classical counterpart (see more in Proposition \ref{classcomputationcyc}). In special cases, the cohomology pairings can be computed from geometric formulas that regularize integral formulas in complex analysis, and the cocycles detect the H\"{o}lder exponent of $C^{\alpha}$. Our techniques extend beyond classical geometry to noncommutative $\theta$-deformations, with the noncommutative torus as a key example.\\

Two spaces of functions are central to this paper. We recall their definitions.

\begin{deef*}[Lipschitz and H\"{o}lder algebras]
\label{lipdef}
Let $(X,\rd)$ be a compact metric space. We define the Lipschitz algebra of $(X,\rd)$ by
$$\Lip(X,\rd):=\{f\in C(X):\,\exists C>0, \mbox{   s.t.   } |f(x)-f(y)|\leq C\rd(x,y)\quad\forall x,y\in X\}.$$
We define $|f|_{\Lip}$ as the optimal constant $C$ in this definition and equip $\Lip(X)$ with the Banach algebra norm $\|f\|_{\Lip}:=\|f\|_{C(X)}+|f|_{\Lip}$. For $\alpha\in (0,1)$, we define the H\"{o}lder algebra with exponent $\alpha$ by
$$C^\alpha(X,d):=\Lip(X,\rd^\alpha).$$
If the metric $\rd$ is understood from context, we simply write $\Lip(X):=\Lip(X,\rd)$ and $C^\alpha(X):=C^\alpha(X,\rd)$. If $X$ is a Riemannian manifold, it is tacitly assumed that $\rd$ is the geodesic distance.
\end{deef*}

\begin{deef*}[Weak Schatten ideals]
\label{defofweaklp}
Let $\He$ denote an infinite-dimensional separable Hilbert space and $\Ko(\He)$ the $C^*$-algebra of compact operators. For $p\in [1,\infty)$, we define the weak Schatten ideal of exponent $p$ by 
$$\mathcal{L}^{p,\infty}(\He):=\{T\in \Ko(\He): \exists C>0, \mbox{   s.t.   } \mu_k(T)\leq C(1+k)^{-1/p}\quad\forall k\geq 0\}.$$
Here $(\mu_k(T))_{k\in \N}$ denotes a decreasing enumeration of the singular values of $T$. The optimal constant $C$ defines a quasi-Banach norm on $\mathcal{L}^{p,\infty}(\He)$ (for $p>1$ there is an equivalent norm, see Proposition \ref{equinorm}). For a singular state $\omega\in (\ell^\infty/c_0)^*$, see Definition \ref{definoofsingularsat} on page \pageref{definoofsingularsat}, we let 
$$\tra_\omega:\mathcal{L}^{1,\infty}(\He)\to \C,$$
denote the associated Dixmier trace (for details see Definition \ref{definoofdix} on page \pageref{definoofdix}). 
\end{deef*}

\begin{remark}
The computation of Dixmier traces is highly non-trivial in the absence of a Weyl law. An element $G\in \mathcal{L}^{1,\infty}(\He)$ is called measurable if $\tra_\omega(G)$ is independent of $\omega$. In particular, an operator $G\in \mathcal{L}^{1,\infty}(\He)$ satisfying a Weyl law $\lambda_k(G)=ck^{-1}+o(k^{-1})$ is measurable and $\tra_\omega(G)=c$.
\end{remark}

\subsection{Main results}

The computation of Dixmier traces hinges on a precise understanding of the mapping properties in a suitable Sobolev scale. A crucial technical ingredient of this article therefore says that, if $Q$ is an operator of order $0$ and $a \in C^{\alpha}$, $[Q,a]$ shares some properties with an operator of order $- \alpha$:

\begin{thm*}
\label{introsobregthm}
Let $\alpha\in(0,1)$ and $s\in (-\alpha,0)$. If $M$ is an $n$-dimensional closed Riemannian manifold and $Q\in \Psi^0(M)$, there is a constant $C=C(Q,\alpha,s)>0$ such that whenever $a\in C^\alpha(M)$ then
\begin{enumerate}
\item[a)] $[Q,a]:W^s(M)\to W^{s+\alpha}(M)$ satisfies $\|[Q,a]\|_{W^s(M)\to W^{s+\alpha}(M)}\leq C\|a\|_{C^\alpha(M)}$,
\item[b)] $\|[Q,a]\|_{\mathcal{L}^{n/\alpha,\infty}(L^2(M))}\leq C\|a\|_{C^\alpha(M)}$. 
\end{enumerate}
\end{thm*}

Here $W^s(M)$, $s\in \R$, denotes the Sobolev scale on the Riemannian manifold $M$. Theorem \ref{sobregthm} (on page \pageref{sobregthm}) proves part a), and Corollary \ref{weakest} (on page \pageref{weakest}) proves that part b) follows from part a) and the Weyl law for elliptic operators. Related results in the literature and their proofs are not sharp; e.g.~from \cite{abels,marschallcpde} one gains less than $\alpha$ derivatives. We remark that part b) follows from the stronger results in \cite{rochbergsemmes}, even for $a$ in the Besov space $B^\alpha_{n/\alpha,\infty}(M)\supseteq C^\alpha(M)$.

\begin{remark*}
\label{introsobregrem}
A consequence of Theorem \ref{introsobregthm} is that whenever $\alpha_1+\cdots +\alpha_k\geq n$, $Q_1, \dots, Q_k\in \Psi^0(M,E)$ (for some vector bundle $E$) and $\omega\in (\ell^\infty/c_0)^*$ is a singular state, the linear mapping 
\begin{align*}
\mathfrak{f}_\omega: \Bo(L^2(M,E))\hat{\otimes} C^{\alpha_1}(M)\hat{\otimes}\cdots &\hat{\otimes}C^{\alpha_k}(M)\to \C\\
&(T,a_1,\dots, a_k) \mapsto \tra_\omega (T[Q_1,a_1]\cdots [Q_k,a_k]),
\end{align*}
is continuous. Here $\hat{\otimes}$ denotes the projective tensor product. Since $\tra_\omega$ is singular, $\mathfrak{f}_\omega(T,a_1,\dots, a_k)=0$ if $\alpha_1+\cdots +\alpha_k>n$. Explicit formulas for $\mathfrak{f}_\omega$ and their geometric interpretations form the main motivation for the paper.
\end{remark*}

The sharp estimate in Theorem \ref{introsobregthm}, part a), together with a refinement of the results in \cite[Chapter 11.2]{sukolord}, allows the computation of Dixmier traces as expectation values in certain orthonormal bases. The resulting formula relates the spectrally defined Dixmier trace to a basis dependent expression, analogously to the Lidskii trace formula; we call it the ordered Lidskii formula. It generalizes the results of \cite[Chapter 11.2]{sukolord}, although the proofs are based on the same technical lemma (see Lemma \ref{lemma11210} below). The formula can be found in Theorem \ref{vmodthm} (on page \pageref{vmodthm}). The ordered Lidskii formula is the basis towards explicit computations of Dixmier traces of products of commutators. In particular, if $M=S^1$ we obtain a bilinear expression of the logarithmic divergence in Brezis' formula for the winding number \cite{kahanewindandfour} on $C^{1/2}(S^1)$. 

\begin{thm*}\label{introszegothm}
Let $P$ be the Szeg\"{o} projection on $S^1$. For $a,b\in C^{1/2}(S^1)$,
\begin{align*}
\tra_\omega(P[a,P][P,b])&=\lim_{N\to \omega} \frac{1}{\log(N)}\sum_{k=0}^Nk\cdot a_k b_{-k}=\lim_{N\to \omega} \lim_{r\nearrow 1}\int_{S^1\times S^1} a_+(\bar{\zeta})b_-(z)k_N(rz,\zeta)\rd \zeta\rd z\ ,
\end{align*}
where $a_+ = Pa = a - a_-$ and
$$k_N(z,\zeta)=\frac{1}{\log (N)}\cdot \frac{1-(z\zeta)^{N+1}}{(1-z\zeta)^2}.$$
\end{thm*}

If $a\in C^{1/2}(S^1,\C^\times)$ and $b=a^{-1}$, $\tra_\omega((2P-1)[P,a][P,a^{-1}])=0$ by Remark \ref{commentsonwinding}. This vanishing result expresses the fact that the logarithmic divergence in Brezis' formula for the winding number of a $C^{1/2}$-function vanishes; there is no obstruction in extending the winding number from $C^\infty$ to $C^{1/2}$ (it even extends by continuity to all continuous functions). We refer to Theorem \ref{computindixprod} and Proposition \ref{inteformcirc} for the formula involving Fourier modes and the integral formula, respectively.\\

A precise study of Dixmier traces for the case of generalized Weierstrass functions illustrates the rich spectral behavior of nonsmooth commutators. The generalized Weierstrass function associated with the parameters $\alpha \in (0,1)$, $1<\gamma \in \N$ and $c \in \ell^\infty(\N)$ is given by the Fourier series
\begin{equation*}
W_{\alpha,\gamma,c}(e^{i \theta}):=2\sum_{n=0}^\infty \gamma^{-\alpha n}c_n\cos(\gamma^n \theta) \in C^\alpha(S^1).
\end{equation*}
We consider a set of sequences that define Hankel measurable operators,
$$\mathfrak{hms}_{k,\gamma}:=\{c\in \ell^\infty(\N): \; [P,W_{1/2k,\gamma,c}]^{2k}\in \mathcal{L}^{1,\infty}(L^2(S^1)) \quad \mbox{is measurable}\}.$$
Corollary \ref{hermajestyscorollary} give many natural examples of non-measurable operators on $S^1$: 
\begin{cor*}
\label{hermajestyscorollaryintro}
The set $\mathfrak{hms}_{1,\gamma}$ equals
$$\mathfrak{hms}_{1,\gamma}=\left\{c=(c_n)_{n\in \N}\in \ell^\infty(\N): \; \lim_{N \to \infty} \frac{1}{N}\sum_{n=0}^{N} c_n^2\ \ \text{exists}\right\}.$$
In particular, the inclusion $\mathfrak{hms}_{1,\gamma}\subseteq \ell^\infty(\N)$ is strict and does not depend on $\gamma$.
\end{cor*}
The Dixmier traces of products of commutators of the Szeg\"{o} projection $P$ with generalized Weierstrass functions $W_{\alpha,\gamma,c}$ can be computed. The results indicate a wide range of asymptotic spectral behavior beyond Weyl's law. See more in Remark \ref{weirdsequences} and the discussion preceding it.\\

In higher dimensions, when $M$ is an $n$-dimensional Riemannian manifold, refining arguments of \cite{sukolord} we obtain a variant of Connes' residue trace formula: For suitable $G\in \mathcal{L}^{1,\infty}(L^2(M))$ 
\begin{align*}
\tra_\omega(G)&=\lim_{N\to \omega}\frac{(2\pi)^{-n}}{\log(N)} \int_M\int_{T^*_xM, |\xi|<N^{1/n}} p_G(x,\xi)\rd x\rd \xi\\
&=\lim_{N\to \omega}\frac{N}{\log(N)}\int_{\textnormal{Diag}_{\epsilon}} k_G(x,y)\phi^*\rho_N(x,y)\rd x\rd y.
\end{align*}
For details and notation, see Theorem \ref{connesformulabitches}. When $k_G$ admits a polyhomogeneous expansion in $x-y$, the limit exists, and therefore $G$ is measurable; for a product of commutators $[P,a]$ this can be verified from properties of the function $a$. Using deep results of Rochberg and Semmes \cite{rochbergsemmes}, the residue trace formula generalizes, in Theorem \ref{ctfforlip}, the usual expression for commutators with $C^\infty$ functions to Lipschitz functions.\\

In the remaining sections, our results are extended to noncommutative tori and (sub-Riemannian) contact manifolds. The analogue of Theorem \ref{introsobregthm} for noncommutative tori is proven in Theorem \ref{sobregtheta}. For sub-Riemannian manifolds, the analogue of Theorem \ref{introsobregthm} was proven in \cite[Section 4]{gimpgoff} for functions Lipschitz in the Carnot-Caratheodory metric. We note that, as a consequence, most of the computational tools are available even if they prove difficult to wield to the perfection that produces explicit formulas. Adapting the results of Rochberg-Semmes to the sub-Riemannian setting, we also extend geometric formulas of Englis-Zhang \cite{engzh} from $C^\infty$ to Lipschitz functions.

\subsection{Spectral asymptotics of nonsmooth operators}

Spectral properties of commutators have also been studied from different perspectives, for example in the areas of harmonic analysis and partial differential equations. When $M=S^1$, the spectral properties of Hankel operators $Pa(1-P)=P[P,a]$ have been extensively explored. The book \cite{peller} by Peller contains detailed information about Schatten and weak Schatten properties. More generally, $\mathcal{L}^{p,q}$-properties were studied in \cite[Chapter 6.4]{peller}. As an example of current work, we refer to \cite{pushya}. 

In dimension greater than one, we mention the recent interest in the spectral behavior of nonsmooth pseudo-differential and Green operators, e.g.~\cite{grubb}, from the view of partial differential equations. They extend classical formulas from the smooth setting to a finite number of derivatives, building on a long history around the spectral asymptotics of integral operators, combining works by Birman and Solomyak \cite{birsol, birsolinterpolation} with precise pseudo-differential techniques. 

Unlike in these works, we are interested in the new phenomena under minimal regularity assumptions below Lipschitz, when the spectral behavior is governed by singularities rather than semi-classical notions like the Poisson bracket. The computations of Dixmier traces give meaning to ``averaged" asymptotic properties - finer information than a growth bound on eigenvalues. The study of such growth bounds in \cite{rochbergsemmes} exhausts the latter problem in the Riemannian setting.

From a different point of view, harmonic analysts have long studied the boundedness of nonsmooth commutators, rather than their spectral behavior, going back to Calderon \cite{calderonone} and Coifman-Meyer \cite{ceemtwo}. See \cite[Chapter 3.6]{nlin} for a recent overview of  results. As noted in \cite{gimpgoff}, the stronger, sharp Theorem \ref{introsobregthm}, part a) is needed for the investigation of Dixmier traces. Moreover, the approach using operator norm bounds to study spectral behavior in \cite{gimpgoff} automatically gives rise to mapping properties similar to Theorem \ref{introsobregthm}, part a), though not to computations of Dixmier traces.

\subsection*{Contents of the paper}

This paper is organized as follows. In Section \ref{introseconriemannain} we introduce Dixmier traces and the ingredients to compute them for commutators, especially on Riemannian manifolds. The Sobolev mapping properties of commutators and weak Schatten estimates in Theorem \ref{introsobregthm} are established. They lead to a $V$-ordered Lidskii formula for these operators in Theorem \ref{vmodthm} and Proposition \ref{prodofcomm}. Also the residue trace formula is derived. For $M=S^1$, Section \ref{hankelsection}.1 uses these tools to prove the Fourier and geometric integral formulas  of Theorem \ref{introszegothm}, while Section \ref{hankelsection}.2 analyzes the Dixmier traces of commutators with Weierstrass functions in detail, including explicit formulas, exotic geometric cocycles and the characterization of measurable operators from Corollary \ref{hermajestyscorollaryintro}. Section \ref{sectiontheta} extends the computational tools to higher-dimensional tori and $\theta$-deformations. The final Section \ref{sectionformerlyknownas9} generalizes the analysis to sub-Riemannian $H$-manifolds, with contact manifolds as a key example. The main results are a Connes-type formula on the Hardy space in Theorem \ref{contform} and the extension of Dixmier trace formulas from smooth to Lipschitz functions in Theorem \ref{extedi}.

\large
\section{Dixmier traces of nonsmooth operators on Riemannian manifolds}
\label{introseconriemannain}
\normalsize

In this section we address commutators of operators appearing on the Riemannian manifolds. The aim is to study the spectral theory of nonsmooth pseudo-differential operators by means of computing Dixmier traces: we focus on products of commutators as motivated, in particular, by Theorem \ref{introszegothm} and Remark \ref{introsobregrem} on page \pageref{introszegothm}. 

In this special case, several results can be obtained by combining and refining theorems from the literature; in Subsection \ref{subsectiononsobmapping} we combine results of Taylor \cite{nlin,toolspde,hbcom}. The technical framework for computing Dixmier traces is presented in Subsection \ref{dixtracesmodsection}; it is an adapted variant of the results by Lord-Sukochev-Zanin \cite{sukolord} and the preceding work by Kalton-Lord-Potapov-Sukochev \cite{kaloposu}.

\subsection{Sobolev mapping properties}
\label{subsectiononsobmapping}
A crucial technical tool are, as in the later sections of the paper, sharp mapping properties in the Sobolev scale $W^s(M)$, $s\in \R$, of a closed Riemannian manifold $M$. 

\begin{thm}
\label{sobregthm}
Let $\alpha\in(0,1)$ and $s\in (-\alpha,0)$. If $M$ is a Riemannian manifold and $Q\in \Psi^0(M)$, there is a constant $C=C(Q,\alpha,s)>0$ such that whenever $a\in C^\alpha(M)$ then
\begin{align*}
[Q,a]:W^s(M)\to &W^{s+\alpha}(M)\quad\mbox{is continuous with}\\
&\|[Q,a]\|_{W^s(M)\to W^{s+\alpha}(M)}\leq C\|a\|_{C^\alpha(M)}.
\end{align*}
\end{thm}
\begin{rm}
From related results in the literature and their proofs, e.g.~\cite{abels,marschallcpde}, one gains less than $\alpha$ derivatives.
\end{rm}
We prove the theorem using Littlewood-Paley theory. For details regarding Littlewood-Paley theory, see \cite{toolspde}. Let us recall the basic idea. 
\label{lpdiscussion}
Locally, the $n$-dimensional manifold $M$ is modeled on $\R^n$. Choose a positive function $\varphi\in C^\infty(\R^{n})$ with $\supp(\varphi)\subseteq B(0,2)$ and $\varphi=1$ on $B(0,1)$. We define the functions
$$\varphi_{j}(x):=\varphi(2^{-j}x),\quad j\in \N.$$
We also define $\phi_{0}:=\varphi=\varphi_{0}$ and $\phi_{j}:=\varphi_{j}-\varphi_{j-1}$ for $j\geq 1$. Since $\varphi_{j}\to 1$ uniformly on compacts, $(\phi_{j})_{j\in \N}$ is a locally finite partition of unity for $\R^{n}$. We call such a partition of unity a Littlewood-Paley partition of unity. We consider the smoothing operators 
$$\Phi_{j}:=\phi_j(D)=Op(\phi_{j}).$$ 
For a Schwartz function $f\in \mathcal{S}(\R^{n})$, these operators are defined from the Fourier transform $\hat{f}$ by 
$$\Phi_{j}f(x)=\int_{\R^n} \phi_j(\xi)\hat{f}(\xi)\e^{ix\xi}\rd \xi=f* \check{\phi}_j(x),$$ 
where $\check{g}(x):=\hat{g}(-\xi)$. We call $(\Phi_j)_{j\in \N}$ a Littlewood-Paley decomposition associated to $(\phi_{j})_{j\in \N}$. Littlewood-Paley theory characterizes the scale of Besov spaces completely, see for instance \cite{abelsbook}, and therefore provides a natural tool to study H\"older and Sobolev spaces.

\begin{proof}[Proof of Theorem \ref{sobregthm}]
The statement of the theorem is of a local nature. Indeed, we can after an argument involving a partition of unity assume that we are dealing with $Q$ and $a$ compactly supported in a neighborhood $U\subseteq M$ diffeomorphic to an open subset of $\R^n$. Since Sobolev spaces on closed manifolds are independent of choice of Laplacian, we can even assume that $Q$ and $a$ are compactly supported in $\R^n$. We let $(\phi_k)_{k=0}^\infty$ denote a Littlewood-Paley partition of unity of $\R^n$ for the remainder of the proof. Associated with this Littlewood-Paley decomposition there are operators $T_f:=\sum_{k=0}^\infty f_k\psi_k(D)$,  where $f_k:=\sum_{j=0}^{k-2} \psi_j(D)(f)$, and $R_f=\sum_{|j-k|\leq 2} \psi_j(D)(f)\psi_k(D)$ for a function $f$. It follows from \cite[Proposition 7.3, Chapter I]{toolspde} that $[Q,T_a]:W^s(M)\to W^{s+\alpha}(M)$ is continuous for $a\in C^\alpha(M)$, and norm continuously depending on $a$, for any $s\in \R$. The operator
$$\varrho(a):W^s(M)\to W^{s+\alpha}(M), \quad\varrho(a)u:=R_au+T_ua$$
is continuous for $s\in (-\alpha,0)$ by \cite[Proposition 3.2.A]{nlin}. In summary, because $Q$ acts continuously on all Sobolev spaces, we have that the operator
$$[Q,a]=[Q,T_a]+Q\varrho(a)-\varrho(a)Q:W^s(M)\to W^{s+\alpha}(M)$$
is continuous for $s\in (-\alpha,0)$.
\end{proof}

\begin{cor}
\label{weakest}
For an $n$-dimensional Riemannian manifold $M$, $Q\in \Psi^0(M)$ and $a\in C^\alpha(M)$ we can estimate
\[\|[Q,a]\|_{\mathcal{L}^{n/\alpha,\infty}(L^2(M))}\leq C\|a\|_{C^\alpha(M)}.\]
\end{cor}

\begin{proof}
Take a Laplacian $\Delta$ on $M$. We can write 
$$[Q,a]=(1+\Delta)^{-(s+\alpha)/2}(1+\Delta)^{(s+\alpha)/2}[Q,a](1+\Delta)^{-s/2}(1+\Delta)^{s/2}.$$
The operator $(1+\Delta)^{(s+\alpha)/2}[Q,a](1+\Delta)^{-s/2}$ has a bounded extension to $L^2(M)$ by Theorem \ref{sobregthm}. Therefore, the Weyl law implies that
\begin{align*}
\|[Q,a]\|_{\mathcal{L}^{n/\alpha,\infty}(L^2(M))}\leq C\|(1+&\Delta)^{(s+\alpha)/2}[Q,a](1+\Delta)^{-s/2}\|_{\Bo(L^2(M))}\cdot\\
&\cdot\|(1+\Delta)^{-(s+\alpha)/2}\|_{\mathcal{L}^{n/(s+\alpha),\infty}(L^2(M))}\|(1+\Delta)^{s/2}\|_{\mathcal{L}^{n/|s|,\infty}(L^2(M))}\\
&\qquad\qquad\qquad\qquad\leq C\|a\|_{C^\alpha(M)}.
\end{align*}
\end{proof}

\begin{remark}
\label{weakscrem}
Estimates of the type in Corollary \ref{weakest} were studied for Heisenberg operators on contact manifolds in \cite{gimpgoff}. In practice, the estimates in Corollary \eqref{weakest} are used to estimate singular values:
$$\mu_k([Q,a])\leq C\|a\|_{C^\alpha(M)}k^{-\frac{\alpha}{n}},$$
where $n=\dim(M)$. 
\end{remark}

\begin{remark}
In the limit cases $s=-\alpha$ or $s=0$, the statement of Theorem \ref{sobregthm} does not hold. It is however known to hold also in the limit cases for Lipschitz functions, i.e. for any $Q\in \Psi^0(M)$, there is a $C(Q)>0$ such that $\|[Q,a]\|_{W^s(M)\to W^{s+1}(M)}\leq C(Q)\|a\|_{\Lip(M)}$ for any $s\in [-1,0]$. This is the content of \cite[Proposition 3.6.B]{nlin}. 
\end{remark}

\begin{remark}
The proof of Theorem \ref{sobregthm} is a to a large extent inspired by the content of the unpublished manuscript \cite{hbcom}. 
\end{remark}

\subsection{Dixmier traces of products}
\label{dixtracesmodsection}

In this subsection, we will study how \emph{some} Sobolev regularity gives a formula for Dixmier traces that we call \emph{the ordered Lidskii formula}. This formula was first obtained in \cite{kaloposu} and studied further in the book \cite{sukolord}. Combining the ordered Lidskii formula with Theorem \ref{sobregthm} will provide us with explicit formulas for the Dixmier traces of products of commutators.

First we recall some well known notions that can be found in for instance \cite{sukolord}. The weak Schatten class $\mathcal{L}^{p,\infty}(\He)$ was defined in Definition \ref{defofweaklp} (see page \pageref{defofweaklp}). 

The following result follows from \cite[Chapter 1.7 and Theorem 2.5]{simon}.

\begin{prop}
\label{equinorm}
The function $T\mapsto \sup_{k\in \N} k^{1/p}\mu_k(T)$ defines a quasinorm making $\mathcal{L}^{p,\infty}(\He)$ into complete metric space. For $p>1$, there is an equivalent norm making $\mathcal{L}^{p,\infty}(\He)$ into a Banach space for $p>1$. 
\end{prop}

\begin{deef}
\label{definoofsingularsat}
A functional $\phi$ on $\ell^\infty(\N)$ is called singular if it factors through $\ell^\infty(\N)/c_0(\N)$, that is, $\phi\in \ell^\infty(\N)^*$ and  $\phi|_{c_0(\N)}=0$. An extended limit is a singular state $\omega$ on $\ell^\infty(\N)$.
\end{deef}

\begin{remark}
If $\omega$ is an extended limit, i.e. $\omega|_{c_0(\N)}=0$ and $\omega(1)=1$, then $\omega(c)=\lim_{n\to \infty} c_n$ whenever $c$ is a convergent sequence. This fact motivates the terminology ``extended limit'' as any singular state is an extension of the limit functional defined on the closed subspace of convergent sequences. The space of singular functionals $(\ell^\infty(\N)/c_0(\N))^*$ is a Banach space and existence of extended limits follows from the Hahn-Banach theorem (the non-separable case). For a sequence $c=(c_n)_{n\in \N}\in \ell^\infty(\N)$ we use the suggestive notation:
$$\lim_{n\to \omega}c_n:=\omega(c).$$
\end{remark}

\begin{deef}
\label{definoofdix}
The Dixmier trace $\tra_\omega:\mathcal{L}^{1,\infty}(\He)\to \C$ associated with a singular state $\omega$ is constructed as the linear extension of
\begin{equation}
\label{defonfdix}
\tra_\omega(G):=\lim_{N\to \omega} \frac{\sum_{k=0}^N \mu_k(G)}{\log(N)}, \quad\mbox{defined for $G\geq 0$}.
\end{equation}
\end{deef}

\begin{remark}
The reader should beware of the following technical pitfall: the expression \eqref{defonfdix} is defined for $G$ in the Lorentz ideal $\mathcal{M}_{1,\infty}:=\{G: \sum_{k=0}^N \mu_k(G)=O(\log(N))\}$. However, for the expression \eqref{defonfdix} to be linear on $\mathcal{M}_{1,\infty}$ it is necessary that $\omega$ is dilation invariant, for details regarding this distinction, see \cite[Lemma 9.7.4]{sukolord}.
\end{remark}

\begin{remark}
\label{connesdixremark}
A proper subset of the Dixmier traces is given by the Connes-Dixmier traces (cf. \cite{c,pietsch,dixmieragain}). A \emph{Connes-Dixmier trace} is defined from an extended limit of the form $\omega=M^*\phi$, where $\phi$ is a singular functional and
$$M:\ell^\infty(\N)\to \ell^\infty(\N),\quad (Mx)_k:=\frac{\sum_{l=0}^k\frac{1}{l+1}x_l}{\log(k+2)}.$$
\end{remark}

For computational purposes, one introduces the notion of $V$-modulated operators. This notion was used in \cite{kaloposu} for the computation of Dixmier traces; it can also be found in \cite[Definition 11.2.1]{sukolord}. We will make use of a minor generalization that we call ``weakly $V$-modulated operators". We say that a bounded operator $V\in \Bo(\He)$ is \emph{strictly positive} if $V$ is positive and $V\mathcal{H}\subseteq \He$ is dense.

\begin{deef}
\label{defsofmod}
Assume that $V$ is a positive operator on a Hilbert space $\He$ and let $G\in \Bo(\He)$ be a bounded operator.
\begin{itemize}
\item We say that $G\in \Bo(\He)$ is $V$-modulated if $\sup_{t>0} t^{1/2}\|G(1+tV)^{-1}\|_{\mathcal{L}^2(\He)}<\infty$, where $\mathcal{L}^2(\He)$ denotes the Hilbert-Schmidt ideal of operators. 
\item Assume also that $V$ is strictly positive and belongs to $\mathcal{L}^{p,\infty}(\He)$ for some $p\in [1,\infty)$. If the densely defined operator $GV^{-1}$ extends to a bounded operator in $\mathcal{L}^{p/(p-1),\infty}(\He)$, we say that $G$ is weakly $V$-modulated.
\end{itemize}
\end{deef}

The condition to be a $V$-modulated operator is as the name suggests stronger than that of being a weakly $V$-modulated operator. This fact follows from the following result. We refer its proof to \cite{sukolord}. 

\begin{prop}[Lemma 11.2.9 in \cite{sukolord}] 
If $V\in \mathcal{L}^{p,\infty}(\He)$, for $p>2$, is strictly positive and $G$ is $V^p$-modulated then $GV^{-1}\in \mathcal{L}^{p/(p-1),\infty}(\He)$. In particular, under said assumptions, $G$ is weakly $V$-modulated if it is $V^p$-modulated.
\end{prop}

\begin{deef}
Let $V\in \mathcal{L}^{1,\infty}(\He)$ be strictly positive and let $(e_l)_{l\in \N}$ denote \emph{the} ON-basis associated with $V$ ordered according to decreasing size of the eigenvalues. For $Q\in \Bo(\He)$ we define the sequence $\mathrm{Res}(Q)=(\mathrm{Res}(Q)_N)_{N\in \N}$ by
$$\mathrm{Res}(Q)_N:=\frac{\sum_{l=0}^N\langle Ge_l,e_l\rangle}{\log(N+2)}.$$
\end{deef}

The residue sequence $\mathrm{Res}(G)$ can be used to compute Dixmier traces assuming $G$ is modulated. The proof of the following theorem is found in \cite{sukolord}.

\begin{thm}[Corollary 11.2.4 in \cite{sukolord}]
If $V\in \mathcal{L}^{1,\infty}(\He)$ and $G$ is $V$-modulated then $\mathrm{Res}(G)\in \ell^\infty(\N)$ and the Dixmier trace can be computed by means of {\bf the $V$-ordered Lidskii formula}
\begin{equation}
\label{losuzaformula}
\tra_\omega(G)=\omega(\mathrm{Res}(G))\equiv \lim_{N\to \omega}\frac{\sum_{l=0}^N\langle Ge_l,e_l\rangle}{\log(N+2)}.
\end{equation}
\end{thm}

\begin{remark}
The terminology ``$V$-ordered Lidskii formula" comes from the fact that for a positive trace class operator $G$, the Lidskii theorem relates the expectation values $\sum_{l=0}^\infty \langle Ge_l,e_l\rangle$ with the spectral expression $\sum_{l=0}^\infty \mu_k(G)$. It should not be confused with the Lidskii formula for Dixmier traces \cite[Theorem 1]{lidskiitypesesuza}. 
\end{remark}

We will need to use the full power in the proofs of the results in \cite[Chapter 11]{sukolord} to obtain an ordered Lidskii formula as in \eqref{losuzaformula} for the larger class of operators that we are interested in - the weakly modulated operators. We use the full strength in \cite[Lemma 11.2.10]{sukolord}. For the convenience of the reader, we recall this result here. Following the notation of \cite{sukolord}, for a compact operator $G$, we let $(\lambda_k(G))_{k\in \N}\subseteq \C$ denote an enumeration of the non-zero eigenvalues of $G$ ordered so that $(|\lambda_k(G)|)_{k\in \N}$ is a decreasing sequence.

\begin{lem}[Lemma 11.2.10, \cite{sukolord}]
\label{lemma11210}
Let $V\in \mathcal{L}^{p,\infty}(\He)$ be a strictly positive operator for some $p\in [1,\infty)$ with associated $ON$-basis $(e_l)_{l\in \N}$. If $G$ is weakly $V$-modulated, then
$$\sum_{k=0}^n\lambda_k(\mathrm{Re}G)=\sum_{l=0}^n\mathrm{Re}\langle Ge_l,e_l\rangle+O(1)\quad \mbox{as $n\to \infty$}.$$
\end{lem}

We refer the proof of Lemma \ref{lemma11210} to \cite{sukolord}. From it, we deduce the following theorem.

\begin{thm}[The $V$-ordered Lidskii formula]
\label{vmodthm}
Let $V\in \mathcal{L}^{p,\infty}(\He)$ be strictly positive for some $p\in [1,\infty)$. If $G\in \mathcal{L}^{1,\infty}(\He)$ is weakly $V$-modulated (see Definition \ref{defsofmod}), then the $V$-ordered Lidskii formula \eqref{losuzaformula} holds.
\end{thm}

\begin{proof}
For any self-adjoint $G\in \mathcal{L}^{1,\infty}(\He)$, the following Lidskii type formula for Dixmier traces follows from the definition of Dixmier traces:
\begin{equation}
\label{lidskiiiitype}
\tra_\omega(G)=\lim_{N\to \omega}\frac{\sum_{k=0}^N\lambda_k(G)}{\log(N)}.
\end{equation}
Using that $G=\mathrm{Re}(G)-i\mathrm{Re}(iG)$, the theorem follows once noting that
\begin{align*}
\tra_\omega(G)&=\tra_\omega(\mathrm{Re}(G))-i\tra_\omega(\mathrm{Re}(iG))\\
&=\lim_{N\to \omega}\frac{\sum_{k=0}^N\lambda_k(\mathrm{Re}G)-i\lambda_k(\mathrm{Re}iG)}{\log(N)}\\
&=\lim_{N\to \omega}\frac{\sum_{l=0}^N\mathrm{Re}\langle Ge_l,e_l\rangle-i\mathrm{Re}\langle iGe_l,e_l\rangle}{\log(N)}=\lim_{N\to \omega}\frac{\sum_{l=0}^N\langle Ge_l,e_l\rangle}{\log(N)}.
\end{align*}
We used Lemma \ref{lemma11210} in the second equality.
\end{proof}

\begin{remark}
The Lidskii type formula for Dixmier traces \eqref{lidskiiiitype} was proven in \cite[Theorem 7.3.1]{sukolord} (cf. \cite[Theorem 1]{lidskiitypesesuza}) for any operator in the larger Lorentz ideal.
\end{remark}

We now reformulate Theorem \ref{vmodthm} in a context fitting better with products on a closed Riemannian manifold $M$. We will formulate the result in a slightly more general setting. We need some terminology first. Let $D$ be an unbounded self-adjoint operator on $\He$ with resolvent in $\mathcal{L}^{p,\infty}(\He)$ for some $p$. Define $W_D^s:=\Dom(|D|^s)$ for $s\geq 0$. For $s<0$ we define $W_D^s$ by duality. 

\begin{deef}
Let $\epsilon\in (0,p)$ and $s\in (-\epsilon,0)$. We say that $G\in \mathcal{L}^{1,\infty}(\He)$ is $(\epsilon,s)$-factorizable with respect to $D$ if we can write $G=G'G''$ where
\begin{enumerate}
\item $G''$ extends to a continuous operator $W^{s}_D\to W^{s+\epsilon}_D$;
\item $G' \in \mathcal{L}^{p/(p-\epsilon),\infty}(\He)$.
\end{enumerate}
\end{deef}

\begin{lem}
\label{lszformnew}
If $G\in \mathcal{L}^{1,\infty}(\He)$ is $(\epsilon,s)$-factorizable with respect to $D$, $G$ is weakly $V$-modulated for $V:=|i+D|^{s}$.
\end{lem}

The structure of the proof is similar to that of \cite[Lemma 5.11]{gimpgoff}.

\begin{proof}
Since $G' \in \mathcal{L}^{p/(p-\epsilon),\infty}(\He)$, it follows that $G'|i+D|^{-(s+\epsilon)} \in \mathcal{L}^{p/(p+s),\infty}(\He)$. It follows from this observation that
$$GD^{-s}=G'|i+D|^{-(s+\epsilon)}\underbrace{|i+D|^{s+\epsilon}G''|i+D|^{-s}}_{\in \Bo(\He)}\in \mathcal{L}^{p/(p+s),\infty}(\He).$$
Since $V\in \mathcal{L}^{p/|s|,\infty}(\He)$, it follows that $G$ is weakly modulated.
\end{proof}

The next theorem contains a Connes type residue formula in the context of weakly Laplacian modulated operators. Its proof is heavily based on the methods of \cite[Chapter 11.6]{sukolord}. We also re-express the Dixmier trace by means of the operator's $L^2$-kernel. First we need some notations and terminology. 

\begin{deef}
\label{localdifd}
An operator $G\in \Bo(L^2(M))$ is said to be localizable if $\chi G\chi'\in \mathcal{L}^1(L^2(M))$ whenever $\chi,\chi'\in C^\infty(M)$ have disjoint support. 
\end{deef}

We use the notation $k_G\in L^2(M\times M)$ for the $L^2$-kernel associated with a Hilbert-Schmidt operator $G\in \mathcal{L}^2(L^2(M))$. For an $\epsilon>0$, we let 
$$\textnormal{Diag}_\epsilon:=\{(x,y)\in M\times M: \mathrm{d}(x,y)<\epsilon\},$$ 
where $\mathrm{d}$ denotes the geodesic distance. We choose a cut-off function $\chi\in C^\infty_c(\textnormal{Diag}_\epsilon)$, which is $1$ near the diagonal $\textnormal{Diag}\subseteq M\times M$, and a diffeomorphism $\phi: \textnormal{Diag}_\epsilon \xrightarrow{\sim} TM$. It exists for small enough $\epsilon$, since $TM$ is isomorphic to the normal bundle of the inclusion $M\cong \textnormal{Diag}\subseteq M\times M$. The $L^2$-symbol $p_G\in C^\infty(T^*M)$ of $G$ is defined as the Fourier transform of $\phi^*(\chi\cdot k_G)\in L^2_c(TM)$, that is
$$p_G(x,\xi):=\int_{T_xM} (\chi k_G)(\phi(x,v))\e^{-i\xi(v)}\rd v.$$
We define the smooth density $\rho\in C^\infty(TM,\pi^*\Omega)$ by 
\begin{equation}
\label{definrho}
\rho(x,v):=\int_{|\xi|\leq 1} \e^{i\xi(v)}\rd \xi.
\end{equation}
We also set $\rho_N(x,v):=\rho(x,N^{1/n}v)\chi(\phi^{-1}(x,v))$.

\begin{thm}[Connes' residue trace formula]
\label{connesformulabitches}
Let $M$ be a Riemannian $n$-dimensional manifold and $\Delta$ a Laplace type operator on $M$. If $G\in \mathcal{L}^{1,\infty}(L^2(M))$ is localizable, weakly modulated w.r.t a negative power of $1+\Delta$ and satisfies the condition that 
\begin{equation}
\label{thelonemod}
\sup_{k\in \N}\int_M\int_{T^*_xM, 2^k<|\xi|<2^{k+1}} |p_G(x,\xi)|\rd x\rd \xi<\infty,
\end{equation}
then
\begin{align}
\label{whatheikocallsconnesformula}
\tra_\omega(G)&=\lim_{N\to \omega}\frac{(2\pi)^{-n}}{\log(N)} \int_M\int_{T^*_xM, |\xi|<N^{1/n}} p_G(x,\xi)\rd x\rd \xi\\
\label{kernelsconnesformula}
&=\lim_{N\to \omega}\frac{N}{\log(N)}\int_{\textnormal{Diag}_{\epsilon}} k_G(x,y)\phi^*\rho_N(x,y)\rd x\rd y.
\end{align}
\end{thm}

\begin{remark}
It follows from the proof of \cite[Proposition 11.3.18]{sukolord} that condition \eqref{thelonemod} is automatically satisfied if $G$ is $(1+\Delta)^{-n/2}$-modulated. Moreover, the conclusion \eqref{whatheikocallsconnesformula} for $(1+\Delta)^{-n/2}$-modulated is stated and proved as \cite[Theorem 11.6.14]{sukolord}. We emphasize that the operators of interest in this paper (as in Remark \ref{introsobregrem} on page \pageref{introsobregrem}) are generally {\bf not} Laplacian modulated but we have not been able to prove that they satisfy the condition \eqref{thelonemod}.
\end{remark}

\begin{proof}[Proof of Theorem \ref{connesformulabitches}]
The proof of Formula \eqref{whatheikocallsconnesformula} under the assumptions of Theorem \ref{connesformulabitches} is proven in \cite[Chapter 11.6]{sukolord} once replacing the usage of \cite[Proposition 11.3.18]{sukolord} by the assumption \eqref{thelonemod}. Note that the smooth density $\rho$ from Equation \eqref{definrho} satisfies that $N\rho(N^{1/n}z)=\int_{|\xi|\leq N^{1/n}} \e^{iz\cdot \xi}\rd \xi$. Hence, after going to local coordinates, Plancherel's theorem implies that 
\begin{align*}
\tra_\omega(G)&=\lim_{N\to \omega}\frac{1}{\log(N)} \int_M\int_{T^*_xM, |\xi|<N^{1/d}} p_G(x,\xi)\rd x\rd \xi\\
&=\lim_{N\to \omega}\frac{N}{\log(N)}\int_{\textnormal{Diag}_{\epsilon}} k_G(x,y)\chi(x,y)\rho(N^{1/d}\phi(x,y))\rd x\rd y\\
&=\lim_{N\to \omega}\frac{N}{\log(N)}\int_{\textnormal{Diag}_{\epsilon}} k_G(x,y)\rho_N(\phi(x,y))\rd x\rd y\ .
\end{align*}
\end{proof}

\begin{deef}
We say that an operator $T\in \Bo(L^2(M))$ is finitely localizable if for any $\chi,\chi'\in C^\infty(M)$ with disjoint support, $\chi T\chi'\cdot \mathcal{L}^{1,\infty}(L^2(M))\subseteq \mathcal{L}^1(L^2(M))$. 
\end{deef}

For instance, any operator in the algebra generated by $C(M)$ and $\Psi^0(M)$ is finitely localizable.

\begin{prop}
\label{prodofcomm}
Assume that
\begin{enumerate}
\item $\beta_1,\beta_2,\ldots, \beta_k\in (0,1]$ are numbers such that
$$\beta_1+\beta_2+\cdots +\beta_k=n;$$
\item $T_1,..., T_k\in \Psi^0(M)$;
\item $f_j\in C^{\beta_j}(M)$, for $j=1,2,\ldots, k$;
\item $T_0\in \Bo(L^2(M))$.
\end{enumerate}
Then the operator $G:=T_0[T_1,f_1]\cdots [T_k,f_k]\in \mathcal{L}^{1,\infty}(L^2(M))$ is $(\epsilon,s)$-factorizable, for $\epsilon=\beta_k$ and $s\in (-\beta_k,0)$. Furthermore, if $n>1$ or $\beta_1,\beta_2,\ldots, \beta_k\in (0,1)$, then $G$ is localizable if $T_0$ is finitely localizable.
\end{prop}

\begin{proof}
It is clear that $G$ satisfies the conditions of Lemma $(\epsilon,s)$-modulated factorizable after applying Theorem \ref{sobregthm} to $G':=T_0[T_1,f_1][T_2,f_2]\cdots [T_{k-1},f_{k-1}]$ and $G''=[T_k,f_k]$. To prove that $G$ is localizable under the assumptions above, we note that for $\chi\in C^\infty(M)$, the Leibniz rule implies that $\tilde{G}:=[T_1,f_1]\cdots [T_k,f_k]\in \mathcal{L}^{1,\infty}(L^2(M))$ satisfies
$$[\tilde{G},\chi]=\sum_{j=1}^k [T_1,f_1]\cdots [[T_j,\chi],f_j]\cdots [T_k,f_k].$$
Since $T_j\in \Psi^0(M)$, $[T_j,\chi]\in \Psi^{-1}(M)$ and Russo's theorem \cite{russo} implies that $[[T_j,\chi],f_j]\in\mathcal{L}^{p}(L^2(M))$ for any $j$ and $p\geq\max(n,2)$. If $\beta_j<1$, then $\beta_j/n<1/n$ so $\sum_{k\neq j}\beta_k/n+1/n>1$. It follows that the property $[T_j,\chi]\in \Psi^{-1}(M)\subseteq \mathcal{L}^{n,\infty}(L^2(M))$, independently of $n$, implies that
$$[T_1,f_1]\cdots [[T_j,\chi],f_j]\cdots [T_k,f_k]\in \mathcal{L}^{n/\beta_1,\infty}\cdot\mathcal{L}^{n/\beta_2,\infty} \cdots \mathcal{L}^{n,\infty}\cdots\mathcal{L}^{n/\beta_k,\infty}\subseteq \mathcal{L}^1(L^2(M)).$$
Hence for $n>1$ or $\beta_1,\beta_2,\ldots, \beta_k\in (0,1)$, $\tilde{G}$ commutes with $C^\infty(M)$ up to $\mathcal{L}^1(L^2(M))$. Under these assumptions, we conclude that $G$ is localizable if $T_0$ is finitely localizable.
\end{proof}

\begin{remark}
It is a straightforward consequence of Theorem \ref{sobregthm} that in the setup of Proposition \ref{prodofcomm}, any Dixmier trace of $T_0[T_1,f_1]\cdots [T_k,f_k]$ only depends on:
\begin{enumerate}
\item the classes of $f_j$ in the Banach space $C^{\beta_j}(M)/\overline{\cup_{\epsilon>0}C^{\beta_j+\epsilon}(M)}$, for $j=1,2,\ldots, k$;
\item the principal symbols $\sigma_0(T_1), \sigma_0(T_2)..., \sigma_0(T_k)\in C^\infty(S^*M)$;
\item the class of $T_0$ in $\Bo(\He)/ \cup_{p>0}\mathcal{L}^{p}(L^2(M))$.
\end{enumerate}
\end{remark}

\begin{remark}
Despite having Theorem \ref{connesformulabitches} at hand for the operators appearing in Proposition \ref{prodofcomm}, we do not know of any explicit asymptotics of the right hand side of \eqref{whatheikocallsconnesformula} in terms of the given geometric data unless we assume that $f_1,..., f_k\in C^1(M)$. For $C^1$-functions, $k=n$ produces the only non-trivial Dixmier traces. By continuity, Proposition \ref{prodofcomm} shows that for $n>1$, Dixmier traces for $C^1$-functions $f_1,...,f_k$ can be computed by a formula resembling the Wodzicki residue.
\end{remark}

We can say more about Dixmier traces in the particular case of an operator $G=T_0[T_1,a_1]\cdots [T_n,a_n]$, where $T_j$ are pseudo-differential operators and $a_j\in \Lip(M)$ on an $n$-dimensional manifold $M$.

\begin{thm}
\label{ctfforlip}
Let $\omega$ be an extended limit and $M$ an $n$-dimensional closed Riemannian manifold. Assume that $T_0,T_1,\ldots, T_n\in \Psi^0(M)$ are classical pseudo-differential operators. Consider the two $n$-linear mappings 
\begin{align}
\label{dtraaapp}
C^\infty(M)^{\otimes n}\ni a_1\otimes \cdots \otimes a_n&\mapsto n(2\pi)^n\tra_\omega(T_0[T_1,a_1]\cdots [T_n,a_n]),\\
\label{wresaapp}
C^\infty(M)^{\otimes n}\ni a_1\otimes \cdots \otimes a_n&\mapsto \int_{S^*M} \sigma(T_0)\{\sigma(T_1),a_1\}\cdots \{\sigma(T_n),a_n\},
\end{align}
where $\{\cdot,\cdot\}$ denotes the Poisson bracket. The two $n$-linear functionals in \eqref{dtraaapp} and \eqref{wresaapp} coincides and are both continuous in the $W^{1,n}$-topology. Moreover, for $a_1,\ldots, a_n\in W^{1,n}(M)$
$$\tra_\omega(T_0[T_1,a_1]\cdots [T_n,a_n])= \frac{1}{n(2\pi)^n}\int_{S^*M} \sigma(T_0)\{\sigma(T_1),a_1\}\cdots \{\sigma(T_n),a_n\}.$$
In particular, $T_0[T_1,a_1]\cdots [T_n,a_n]$ is measurable for $a_1,\ldots, a_n\in \Lip(M)$.
\end{thm}

\begin{proof}
It is immediate from the construction that the functional in \eqref{wresaapp} is continuous in the $W^{1,n}$-topology. The functional in \eqref{dtraaapp} is continuous in the $W^{1,n}$-topology because of the following reasons. For $Q\in \Psi^0(M)$, $a\mapsto [Q,a]\in \mathcal{L}^{n,\infty}(L^2(M))$ is continuous in the $\textnormal{Osc}^{n,\infty}$-topology (see \cite{rochbergsemmes}) and $\textnormal{Osc}^{n,\infty}=W^{1,n}$ by \cite[Appendix]{connessulltele}. Since $C^\infty(M)$ is dense in $W^{1,n}(M)$, we can extend both functionals  \eqref{dtraaapp} and \eqref{wresaapp} to $W^{1,n}(M)$ by continuity. The two functionals  \eqref{dtraaapp} and \eqref{wresaapp} coincides on $C^\infty(M)$ using Connes residue trace formula (see Theorem \ref{connesformulabitches}).
\end{proof}

\begin{remark}
What Theorem \ref{ctfforlip} shows is that classicality at the level of Dixmier traces still holds for $T_0[T_1,a_1]\cdots [T_n,a_n]$ with $a_1,\ldots, a_n\in \Lip(M)$. We shall see below in Section \ref{weiersection} that this need not hold in $C^\beta$ for $\beta<1$.
\end{remark}

\large
\section{Hankel operators and commutators on $S^1$}
\label{hankelsection}
\normalsize

In this section, we restrict our attention to $M=S^1$ and $T=P$ -- the Szeg\"o projection. To simplify notation we normalize the volume of $S^1$ to $1$. Recall that the Szeg\"o projection $P\in \Psi^0(S^1)$ is the orthogonal projection onto the Hardy space -- the closed subspace  $H^2(S^1)\subseteq L^2(S^1)$ consisting of functions with a holomorphic extension to the interior of $S^1\subseteq \C$. In terms of the Fourier basis $e_l(z):=z^l$, $l\in \Z$, the Szeg\"o projection $P$ is determined by $Pe_l=e_l$ for $l\in \N$ and $Pe_l=0$ for $l\in \Z\setminus \N$. 

\subsection{Computations for $C^{1/2}$} 
\label{hankelsubsection}
We will in this subsection compute Dixmier traces for operators of the form
$$G=P[P,a][P,b],\quad\mbox{for}\quad a,b\in C^{1/2}(S^1).$$

\begin{prop}
\label{hankelcomp}
Any $a\in C(S^1)$ satisfies
$$[P,a]=Pa(1-P)-(1-P)aP=Pa_+(1-P)-(1-P)a_-P,$$
where $a_+=P(a)$ and $a_-=a-a_+=(1-P)(a)$.
\end{prop}

The proof of Proposition \ref{hankelcomp} is based on simple algebra and is left to the reader. We will refer to both $Pa(1-P)$ and $(1-P)aP$ as Hankel operators with symbol $a$. Hankel operators have been extensively studied in the literature. We mention the standard reference \cite{peller} as a source for further details.

\begin{lem}
\label{henkeldixcomp}
Let $a,b\in C(S^1)$ have Fourier expansions $a= \sum_{k\in \Z} a_ke_k$ and $b= \sum_{k\in \Z} b_ke_k$ (in $L^2$-sense).
It holds that
\begin{align}
\label{pcommab}
\langle P[P,a][P,b]e_l,e_l\rangle_{L^2(S^1)}&=
\begin{cases}
-\sum_{k>l} a_{k}b_{-k},\quad &l\geq 0\\
0, \quad &l<0
\end{cases}\quad\mbox{and}\\
\label{commab}
\langle [P,a][P,b]e_l,e_l\rangle_{L^2(S^1)}&=
\begin{cases}
-\sum_{k>l} a_{k}b_{-k},\quad &l\geq 0\\
-\sum_{k\leq l} a_{k}b_{-k}, \quad &l<0
\end{cases}.
\end{align}
\end{lem}

\begin{proof}
We only prove the identity \eqref{pcommab}. The identity \eqref{commab} is computed in a similar manner. It follows from Proposition \ref{hankelcomp} that
$$P[P,a][P,b]=-Pa_+(1-P)b_-P.$$
Hence $\langle P[P,a][P,b]e_l,e_l\rangle_{L^2(S^1)}=0$ for $l<0$. We also note that $\overline{a_+}=(\overline{a})_-+\overline{a_0}$. For simplicity we assume $a_0=0$ which does not alter the operator $P[P,a][P,b]$ nor the right hand side of Equation \eqref{pcommab}. For $l\geq 0$,
\begin{align*}
\langle P[P,a][P,b]e_l,e_l\rangle_{L^2(S^1)}&=-\langle (1-P)b_-e_l,(\overline{a})_-e_l\rangle_{L^2(S^1)}\\
&=-\left\langle (1-P)\sum_{k=1}^\infty b_{-k}e_{l-k},\sum_{k=1}^\infty \overline{a_{k}}e_{l-k}\right\rangle_{L^2(S^1)}=-\sum_{k>l} a_{k}b_{-k}.
\end{align*}
\end{proof}

Due to the structures appearing in Lemma \ref{henkeldixcomp} we often restrict our attention to computing Dixmier traces of the product of Hankel operators $P[P,a][P,b]=-Pa(1-P)bP=-Pa_+(1-P)b_+P$. The next Proposition shows that it suffices for describing the general picture.

\begin{prop}
\label{tracexompu}
For $a^1,\ldots, a^{2k}\in C^{\frac{1}{2k}}(S^1)$,
\begin{align*}
\tra_\omega ([P,a^1]\cdots [P,a^{2k}])=&(-1)^k\tra_\omega(Pa_+^1(1-P)\cdots (1-P)a_-^{2k}P)\\
&+(-1)^k\tra_\omega((1-P)a_-^1Pa_+^2(1-P)\cdots Pa^{2k}_+(1-P)),
\end{align*}
so in particular, for $a,b\in C^{1/2}(S^1)$,
\begin{align*}
\tra_\omega(P[P,a][P,b])&=\tra_\omega((1-P)[P,b][P,a]),\quad\mbox{and}\\
\tra_\omega([P,a]^2)&=2\tra_\omega(P[P,a]^2)=-2\tra_\omega(Pa_+(1-P)a_-P).
\end{align*}
For $a^1,\ldots, a^{2k+1}\in C^{\frac{1}{2k+1}}(S^1)$,
$$\tra_\omega [P,a^1]\cdots [P,a^{2k+1}]=\tra_\omega P[P,a^1]\cdots [P,a^{2k+1}]=0.$$
\end{prop}

\begin{proof}
The first statement follows from the tracial property of $\tra_\omega$ and that for any $a^1,\ldots, a^{2k}\in C(S^1)$
\begin{align*}
[P,a^1]\cdots [P,a^{2k}]=&(-1)^kPa^1(1-P)a^2P\cdots (1-P)a^{2k}P\\
&\quad+(-1)^k(1-P)a^1Pa^2(1-P)\cdots Pa^{2k}(1-P)\\
=&(-1)^kPa_+^1(1-P)a_-^2P\cdots (1-P)a_-^{2k}P\\
&\quad+(-1)^k(1-P)a_-^1Pa_+^2(1-P)\cdots Pa^{2k}_+(1-P).
\end{align*}
The second statement follows from the tracial property of $\tra_\omega$ and that for any $a^1,\ldots, a^{2k+1}\in C(S^1)$
\begin{align*}
[P,a^1]\cdots [P,a^{2k+1}]=&(-1)^kPa^1(1-P)a^2P\cdots (1-P)a^{2k}Pa^{2k+1}(1-P)\\
&\quad-(-1)^k(1-P)a^1Pa^2(1-P)\cdots Pa^{2k}(1-P)a^{2k+1}P\\
=&(-1)^kPa_+^1(1-P)a^2P\cdots (1-P)a_-^{2k}Pa_+^{2k+1}(1-P)\\
&\quad-(-1)^k(1-P)a_-^1Pa_+^2(1-P)\cdots Pa^{2k}(1-P)a_-^{2k+1}P.
\end{align*}
\end{proof}

\begin{remark}
The identities of Proposition \ref{tracexompu} are standard when studying Connes-Chern characters from finitely summable $K$-homology to cyclic cohomology, see more in \cite[Chapter III]{c}.
\end{remark}

\begin{thm}
\label{computindixprod}
For $a,b\in C^{1/2}(S^1)$ the following computations hold:
\begin{align}
\label{hankelcommprod}
\tra_\omega(Pa(1-P)bP)&=\lim_{N\to \omega} \frac{1}{\log(N)}\sum_{k=0}^Nk\cdot  a_k b_{-k}\\
\label{commprod}
\tra_\omega([P,a][P,b])&=-\lim_{N\to \omega} \frac{1}{\log(N)}\sum_{k=-N}^N|k|\cdot a_k b_{-k}.
\end{align}
\end{thm}

To prove this theorem, we make use of the Littlewood-Paley description of the H\"older continuous functions. We follow the approach to Besov type spaces on $S^1$ in \cite{peller}. We introduce a natural number $\gamma>1$ in the definition to simplify some computations below in Section \ref{weiersection}. Consider the discrete Littlewood-Paley decomposition $(w_{n,\gamma})_{n\in \Z}$ defined as follows. We set $w_{0,\gamma}(z)=1$. 
\label{lptheoryontorus}
For $n>0$, the Fourier coefficients $(\hat{w}_{n,\gamma}(k))_{k\in \Z}$ of $w_{n,\gamma}$ are defined by $\hat{w}_{n,\gamma}(\gamma^n)=1$, $\hat{w}_{n,\gamma}(k)=0$ for $k$ not in $(\gamma^{n-1},\gamma^{n+1})$ and defined as the linear extension of these properties. For $n<0$, we set $w_{n,\gamma}:=\overline{w_{-n,\gamma}}$. The norm $\|\cdot\|_{C^{\alpha,*}(S^1)}$ is given by 
$$\|f\|_{C^{\alpha,*}(S^1)}:=\sup_{n\in \Z}\gamma^{|n|\alpha}\|w_{n,\gamma}*f\|_{L^\infty(S^1)}.$$
By standard arguments, the norm $\|\cdot\|_{C^{\alpha,*}(S^1)}$ is equivalent to the usual norm on $C^\alpha(S^1)$ for $\alpha\in (0,1)$, cf. \cite[Theorem 6.1]{abelsbook}.

\begin{proof}[Proof of Theorem \ref{computindixprod}]
By Proposition \ref{tracexompu}, the identity \eqref{commprod} follows from the identity \eqref{hankelcommprod}. It follows from Theorem \ref{vmodthm} and Proposition \ref{prodofcomm}, using the computation of Lemma \ref{henkeldixcomp}, that 
\begin{align*}
\tra_\omega(Pa(1-P)b)&=\lim_{N\to \omega}\frac{\sum_{l=0}^N\sum_{k>l} a_{k}b_{-k}}{\log(N)}=\lim_{N\to \omega}\frac{\sum_{k=1}^\infty\sum_{l=0}^{\min(k-1,N)} a_{k}b_{-k}}{\log(N)}\\
&=\lim_{N\to \omega}\frac{\sum_{k=1}^Nka_{k}b_{-k}+(N+1)\sum_{k=N+1}^\infty a_{k}b_{-k}}{\log(N)}.
\end{align*}
The theorem follows once proving that $\sum_{k=N+1}^\infty a_{k}b_{-k}=O(N^{-1})$ as $N\to \infty$. Cauchy-Schwarz inequality implies that 
$$\left|\sum_{k=N+1}^\infty a_{k}b_{-k}\right|^2\leq \left(\sum_{k=N+1}^\infty |a_{k}|^2\right)\cdot \left(\sum_{k=N+1}^\infty |b_{-k}|^2\right),$$
hence we can assume that $a=\bar{b}$ and $b_{-k}=\bar{a}_k$. We can also restrict to showing $\sum_{k=2^{j}}^\infty |a_{k}|^2=O(2^{-j})$ as $j\to \infty$. 

Note that $\sum_n\hat{w}_{n,2}(k)=1$ for any $k$, so 
$$|a_k|^2\leq 2\sum_n|\hat{w}_{n,2}(k)a_k|^2.$$
We can conclude from this estimate and Parseval's identity that 
\begin{align*}
\sum_{k=2^{j}}^\infty |a_{k}|^2&\leq 2\sum_{n=j}^\infty \sum_{k=2^{n-1}+1}^{2^{n+1}-1}|\hat{w}_{n,2}(k)a_k|^2=2\sum_{n=j}^\infty\|w_{n,2}*a\|_{L^2(S^1)}^2\\
&\leq2\sum_{n=j}^\infty\|w_{n,2}*a\|_{L^\infty(S^1)}^2 \leq 2\sum_{n=j}^\infty 2^{-n}\|a\|_{C^{1/2,*}(S^1)}^2=2^{-j+2}\|a\|_{C^{1/2,*}(S^1)}^2=O(2^{-j}).
\end{align*}

\end{proof}

\begin{remark}
\label{commentsonwinding}
If $a,b\in H^{1/2}(S^1)$, the operator $[P,a][P,b]$ is trace class and the Lidskii formula for operator traces implies that 
$$\tra([PaP,PbP])=\tra((2P-1)[P,a][P,b])=-\sum_{k=-\infty}^\infty k\cdot a_k b_{-k}=(Da,\tau^*b),$$
where $D=i\rd/\rd\theta$ is the Dirac operator on $S^1$, $(\cdot,\cdot)$ denotes the bilinear pairing $H^{-1/2}(S^1)\times H^{1/2}(S^1)\to \C$ and $\tau(z)=\bar{z}$. An interesting fact is that if $a=b^{-1}\in H^{1/2}(S^1,S^1)$ then $\tra((2P-1)[P,a][P,a^{-1}])=-\deg_{H^{1/2}}(a)\in \Z$ -- the $H^{1/2}$-mapping degree. Hence, the computations of Theorem \ref{computindixprod} can be considered as a Dixmier regularization of the $H^{1/2}$-mapping degree. In \cite{bourgeonkahane}, motivated by $VMO$-mapping degrees, it was proven that if $f\in VMO(S^1,S^1)$ then $Pf\in H^s(S^1)$ for an $s\in (0,1)$ implies that $f\in H^s(S^1)$. For an interesting overview of similar problems, see \cite{kahanewindandfour}. We note that if $a\in C^{1/2}(S^1,\C^\times)$, an argument using the periodicity operator in cyclic cohomology shows that 
\begin{equation}
\label{dixregmappi}
\tra_\omega((2P-1)[P,a][P,a^{-1}])=0.
\end{equation}
One might ask if the vanishing of ``Dixmier regularized mapping degrees" \eqref{dixregmappi} has any consequences on questions similar to those in \cite{kahanewindandfour} for H\"older spaces. 
\end{remark}

\begin{prop}
\label{cyclicprop}
Let $\omega$ be a singular state on $\ell^\infty(\N)$. We define
$$c_\omega(a,b):=\tra_\omega((2P-1)[P,a][P,b]), \quad a,b\in C^{1/2}(S^1).$$
The functional $c_\omega$ defines a continuous cyclic $1$-cocycle on $C^{1/2}(S^1)$.
\end{prop}

The proof of Proposition \ref{cyclicprop} is a short algebraic manipulation which we leave to the reader. We will in Proposition \ref{classcomputationcyc} see that the cohomology class $[c_\omega]\in HC^1(C^{1/2}(S^1))$ is in general non-vanishing and highly dependent on the choice of extended limit $\omega$. We end this subsection with an integral formula for the Dixmier trace of the product of two commutators or Hankel operators; it can be seen as a variation of Theorem \ref{connesformulabitches}.

\begin{prop}
\label{inteformcirc}
For $a,b\in C^{1/2}(S^1)$,
\begin{align*}
\tra_\omega(P[P,a][P,b])&=-\lim_{N\to \omega} \lim_{r\nearrow 1}\int_{S^1\times S^1} a_+(\bar{\zeta})b_-(z)k_N(rz,\zeta)\rd \zeta\rd z,\quad\mbox{and}\\
\tra_\omega([P,a][P,b])&=-\lim_{N\to \omega} \lim_{r\nearrow 1}\int_{S^1\times S^1} \left\{a_+(\bar{\zeta})b_-(z)+b_+(\bar{\zeta})a_-(z)\right\}k_N(rz,\zeta)\rd \zeta\rd z,
\end{align*}
where
$$k_N(z,\zeta)=\frac{1}{\log (N)}\cdot \frac{1-(z\zeta)^{N+1}}{(1-z\zeta)^2}.$$
\end{prop}

\begin{proof}
Assume for simplicity that $a=a_+$ and $b=b_-$, we can also assume $a_+(0)=0$. For $l\geq 0$, we let $P_l$ denote the ON-projection onto the closed linear span of $(e_k)_{k> l}$. We have that
$$\int_{S^1} P_l(a)(z) b(z)\frac{\rd z}{iz}=\sum_{k>l} a_kb_{-k}.$$
The function $P_l(a)$ admits a holomorphic extension to $|z|<1$ where
$$P_l(a)(z)=\frac{1}{i}\int_{S^1}\frac{z^{l+1}\zeta^{-l-1}a(\zeta)\rd \zeta}{\zeta-z}=iz\int_{S^1}\frac{z^{l}\zeta^{l}a(\bar{\zeta})\rd \zeta}{1-z\zeta}.$$
Hence, for $|z|<1$,
$$\sum_{l=0}^NP_l(a)(z)=iz\int_{S^1}\frac{(1-(z\zeta)^{N+1})a(\bar{\zeta})\rd \zeta}{(1-z\zeta)^2}.$$
In particular, for $|z|=1$, $P_l(a)(z)=\lim_{r\nearrow 1}P_l(a)(rz)$. We conclude that
$$\sum_{l=0}^N \sum_{k>l} a_kb_{-k}=\lim_{r\nearrow 1}\int_{S^1}\int_{S^1}\frac{(1-(rz\zeta)^{N+1})a(\bar{\zeta})b(z)\rd \zeta\rd z}{(1-rz\zeta)^2}.$$
The second formula follows from the first using Proposition \ref{tracexompu}.
\end{proof}

\begin{remark}
In fact, the proof of Proposition \ref{inteformcirc} only depends on Lemma \ref{henkeldixcomp} and the results of Section \ref{introseconriemannain}. The integral formula of Proposition \ref{inteformcirc} does, together with a simple Fourier series calculation, imply Theorem \ref{computindixprod}, avoiding any usage of Littlewood-Paley decompositions.
\end{remark}

\begin{remark}
In \cite{engroch}, Dixmier traces of Hankel operators on the Bergman space were considered. The spectral behavior of Hankel operators on the Bergman space differs from that of Hankel operators on the Hardy space. For instance, a smooth function produces a Hankel operator on the Bergman space that need not be more than of weak trace class. A Hankel operator with smooth symbol on the Hardy space has rapid decay in its singular values, so it is of trace class and all its Dixmier traces vanish.
\end{remark}

\begin{remark}
\label{ktwosection}
Following \cite{kaadcomparison}, a numerical invariant $D_\omega:K_2^{alg}(C^{1/2}(S^1))\to \C$ can be constructed from the Dixmier determinant $\det_\omega:K_1^{alg}(\mathcal{L}^{1,\infty}(H^2(S^1)),\Bo(H^2(S^1)))\to \C$ and the short exact sequence of Banach algebras:
\begin{equation}
\label{theoninfseq}
0\to \mathcal{L}^{1,\infty}(H^2(S^1))\to \mathcal{T}_{1/2}(S^1)\to C^{1/2}(S^1)\to 0,
\end{equation}
where $\mathcal{T}_{1/2}(S^1):=PC^{1/2}(S^1)P+\mathcal{L}^{1,\infty}(H^2(S^1))\subseteq \Bo(H^2(S^1))$. This relates to our computations through the identity
$$D_\omega\{a,b\}=\tra_\omega[PaP,PbP]=c_\omega(a,b),$$
where $\{a,b\}\in K_2^{alg}(C^{1/2}(S^1))$ is the Steinberg symbol associated with $a,b\in C^{1/2}(S^1)$ and $c_\omega$ is the cyclic $1$-cocycle from Proposition \ref{cyclicprop}.
\end{remark}

\subsection{Generalized Weierstrass functions and non-classicality}
\label{weiersection}

Our focus in this subsection will be on certain generalized Weierstrass functions and the Dixmier traces of the associated Hankel operators. The main goal being to provide a large set of examples for which the Dixmier traces display a rich behavior. For $\alpha\in (0,1)$, $\gamma\in \N_{>1}$ and a bounded sequence $c=(c_n)_{n\in \N}\in \ell^\infty(\N)$, we define the generalized Weierstrass function
\begin{equation}
\label{otherweinot}
W_{\alpha,\gamma,c}(z):=\sum_{n=0}^\infty \gamma^{-\alpha n} c_n(z^{\gamma^n}+z^{-\gamma^n})=2\sum_{n=0}^\infty \gamma^{-\alpha n}c_n\cos(\gamma^n \theta), \quad\mbox{for}\quad z=\e^{i\theta}.
\end{equation}
Since $W_{\alpha,\gamma,c}$ is defined from an absolutely convergent Fourier series, $W_{\alpha,\gamma,c}$ is continuous. The (non-generalized) Weierstrass function $W_{\alpha,\gamma}$ is obtained by
$$W_{\alpha,\gamma}:=W_{\alpha,\gamma,(1,1,1,1\ldots)}.$$
The function $W_{\alpha,2}$ was studied in for instance \cite[Chapter II]{zygbook}.

\begin{prop}
\label{weierprop}
For any $p\in [1,\infty]$, the construction of generalized Weierstrass functions $W_{\alpha,\gamma,c}$ gives rise to a continuous linear injective mapping
$$\mathfrak{w}_{\alpha,\gamma}:\ell^\infty(\N)\to C^\alpha(S^1)\cap B^\alpha_{p,\infty}(S^1)\cap F^\alpha_{p,\infty}(S^1), \quad c\mapsto W_{\alpha,\gamma,c}.$$
Moreover, if $c$ is an invertible element of $\ell^\infty(\N)$, the inclusion $W_{\alpha,\gamma,c}\in B^\alpha_{p,q}(S^1)$, as well as $W_{\alpha,\gamma,c}\in F^\alpha_{p,q}(S^1)$, holds if and only if $q=\infty$ for any $p\in[1,\infty]$.
\end{prop}

The proof that $W_{\alpha,\gamma,c}\in C^\alpha(S^1)$ is an adaptation of the proof of \cite[Theorem $4.9$, Chapter II]{zygbook}, where the Weierstrass function $W_{\alpha,2}$ was considered. Recall the definition of the Besov space $B^\alpha_{p,q}(S^1)$ and the Triebel-Lizorkin space $F^\alpha_{p,q}(S^1)$ from \cite{peller}. We use the notation of the paragraph succeeding Theorem \ref{computindixprod}. The norm on the Besov space is defined for $p\in [1,\infty]$, $q\in [1,\infty]$ by
$$\|f\|_{B^t_{p,q}(S^1)}:=\left\|\left(\gamma^{|n|t}\|w_{n,\gamma}*f\|_{L^p(S^1)}\right)_{n\in \Z}\right\|_{\ell^q(\Z)}.$$
The norm on the Triebel-Lizorkin space is defined for $p\in [1,\infty)$, $q\in [1,\infty]$ by
$$\|f\|_{F^t_{p,q}(S^1)}:=\left\|\left\|\left(\gamma^{|n|t}w_{n,\gamma}*f(\cdot)\right)_{n\in \Z}\right\|_{\ell^q(\Z)}\right\|_{L^p(S^1)}.$$
For $p=\infty$ and $q\in [1,\infty]$ the norm on the Triebel-Lizorkin space is given by
$$\|f\|_{F^t_{\infty,q}(S^1)}:=\sup_{B\subseteq S^1, |B|=2^{-k}}\left(\frac{1}{|B|}\sum _{n>k} \int_B \gamma^{tnq}|w_{n,\gamma}*f(\e^{i\theta})|^q\rd \theta\right)^{1/q},$$
where the supremum is over all $k$ and all intervals $B$ in $S^1$ of length $|B|=2^{-k}$.

\begin{proof}
Let $h>0$ be small and take $N=N(h)$ such that $\gamma^N h\leq 1$ and $\gamma^{N+1}h >1$. With $z=\e^{i\theta}$, it holds that
\begin{align*}
W_{\alpha,\gamma,c}(z\e^{ih})-W_{\alpha,\gamma,c}(z)=-2&\underbrace{\sum_{n=1}^N \gamma^{-n\alpha}c_n \sin (\gamma^n(\theta+h/2))\sin(\gamma^nh/2) }_{I}\\
&-2\underbrace{\sum_{n=N+1}^\infty \gamma^{-n\alpha}c_n \sin (\gamma^n(\theta+h/2))\sin(\gamma^nh/2) }_{II}
\end{align*}
We estimate
\begin{align*}
|I|\lesssim \sum_{n=1}^N \gamma^{-n\alpha}\gamma^n h=h\sum_{n=1}^N\gamma^{n(1-\alpha)}=h\frac{\gamma^{1-\alpha}-\gamma^{N(1-\alpha)}}{1-\gamma^{1-\alpha}}=h\cdot O(h^{\alpha-1})=O(h^\alpha).
\end{align*}
\begin{align*}
|II|\lesssim \sum_{n=N+1}^\infty \gamma^{-n\alpha} =\frac{\gamma^{-(N+1)\alpha}}{1-\gamma^{-\alpha}}=O(h^\alpha).
\end{align*}
From these estimates, it follows that $W_{\alpha,\gamma,c}\in C^\alpha(S^1)$.

It holds that $w_{n,\gamma}*W_{\alpha,\gamma,c}(z)=\gamma^{-\alpha|n|}c_nz^{\mathrm{sign}(n)\gamma^{|n|}}$. Hence
$$\|W_{\alpha,\gamma,c}\|_{B^t_{p,q}}=\left\|(|c_n|\gamma^{|n|(t-\alpha)})_{n\in \Z}\right\|_{\ell^q(\Z)}.$$
It follows that for $p\in [1,\infty]$, $W_{\alpha,\gamma,c}\in B^t_{p,q}(S^1)$ if $t<\alpha$ and $q\in[1,\infty)$ or $t=\alpha$ and $q=\infty$. If $c$ is invertible in $\ell^{\infty}(\N)$, there is a $\delta>0$ such that $|c_n|\geq \delta$ for all $n$ so the converse holds: $W_{\alpha,\gamma,c}\in B^t_{p,q}(S^1)$ implies that $t<\alpha$ and $q\in[1,\infty)$ or $t=\alpha$ and $q=\infty$. In the same manner, for $p\in [1,\infty)$,
$$\|W_{\alpha,\gamma,c}\|_{F^t_{p,q}}=\left\| \left(\sum_{n} |c_n|\gamma^{q|n|(t-s)}\right)^{1/q}\right\|_{L^p(S^1)}\sim \left\|(|c_n|\gamma^{|n|(t-s)})_{n\in \Z}\right\|_{\ell^q(\Z)}.$$
It follows that for $p\in [1,\infty)$, $W_{\alpha,\gamma,c}\in F^t_{p,q}(S^1)$ if $t<\alpha$ and $q\in[1,\infty)$ or $t=\alpha$ and $q=\infty$. If $c$ is invertible, $W_{\alpha,\gamma,c}\in F^t_{p,q}(S^1)$ implies that $t<\alpha$ and $q\in[1,\infty)$ or $t=\alpha$ and $q=\infty$. For $p=\infty$ and $q\in [1,\infty]$
$$\|W_{\alpha,\gamma,c}\|_{F^t_{\infty,q}}=\sup_k \left(\sum_{j>k} c_n\gamma^{q|n|(t-s)}\right)^{1/q}= \|c_n\gamma^{|n|(t-s)}\|_{\ell^q(\Z)},$$
which is finite if $t<\alpha$ and $q\in[1,\infty)$ or $t=\alpha$ and $q=\infty$ and for $c$ invertible the converse holds.
\end{proof}

Let us turn to a computation of the Dixmier trace of the product of two Hankel operators whose symbol is a generalized Weierstrass function. We use the notation $C:\ell^\infty(\N)\to \ell^\infty(\N)$ for the Cesaro mean defined on $x=(x_n)_{n\in \N}$ by
$$C(x)_N:=\frac{1}{N+1}\sum_{n=0}^{N} x_n.$$
The following proposition follows directly from Theorem \ref{computindixprod}. For $\gamma\in \N_{>1}$, we define $\omega_\gamma(c):=\lim_{\gamma^N\to \omega} c_N$. It is clear from the construction that $\omega\mapsto \omega_\gamma$ defines an injective mapping on $(\ell^\infty(\N)/c_0(\N))^*$.

\begin{prop}
\label{hankcomp}
For $\gamma\in \N_{>1}$ and $c,d\in \ell^\infty(\N)$ it holds that
$$\tra_\omega(P[P,W_{1/2,\gamma,c}][P,W_{1/2,\gamma,d}])=-\frac{\omega_\gamma\circ C(c\cdot d)}{\log(\gamma)}=-\lim_{N\to \omega_\gamma}\frac{1}{N+1}\sum_{n=0}^{N} \frac{c_nd_n}{\log(\gamma)}.$$
In particular,
$$\tra_\omega(P[P,W_{1/2,\gamma,c}]^2)=\tra_\omega((1-P)[P,W_{1/2,\gamma,c}]^2)=-\frac{\omega_\gamma\circ C(c^2)}{\log(\gamma)}.$$
\end{prop}

\begin{remark}
We note that Proposition \ref{hankcomp} implies that Dixmier traces of products of two Hankel operators with symbol being a generalized Weierstrass function can be as wild as the extended limit $\omega$ applied to the image of the Cesaro mean $\im(C:\ell^\infty(\N)\to \ell^\infty(\N))$. Similarly, Connes-Dixmier traces (cf. discussion in the beginning of Section \ref{dixtracesmodsection}) can be as wild as any singular state on $\im(M\circ C:\ell^\infty(\N)\to \ell^\infty(\N))$. This observation follows from the fact that the product $\ell^\infty(\N)\times \ell^\infty(\N)\ni (c,d)\mapsto c\cdot d\in \ell^\infty(\N)$ is surjective.
\end{remark}

For a Banach algebra $\mathcal{A}$ we let $HC^*(\mathcal{A})$ and $HC_*(\mathcal{A})$ denote its cyclic cohomology and cyclic homology, respectively. For details on cyclic theories, see \cite[Chapter III]{c}.

\begin{prop}
\label{classcomputationcyc}
For any non-zero $\omega\in (\ell^\infty(\N)/c_0(\N))^*$, the class $[c_\omega]\in HC^1(C^{1/2}(S^1))$ (see Proposition \ref{cyclicprop}) is non-vanishing. Moreover, 
\begin{equation}
\label{spaceofcocy}
\{[c_\omega]:\omega\in (\ell^\infty(\N)/c_0(\N))^*\}\subseteq HC^1(C^{1/2}(S^1)),
\end{equation}
is an infinite-dimensional subspace.
\end{prop}

\begin{proof}
Take $c,d\in \ell^\infty(\N)$ and define the cyclic homology class $x_{c,d}:=[W_{\gamma,c}^+\otimes W_{\gamma,d}^-]\in HC_1(C^{1/2}(S^1))$. Proposition \ref{hankcomp} shows that
\begin{equation}
\label{comppair}
[c_\omega].x_{c,d}=-\frac{\omega_\gamma\circ C(c\cdot d)}{\log(\gamma)}.
\end{equation}
Hence, if $\omega\neq 0$, then $\omega_\gamma\neq 0$ and $[c_\omega].x_{c,d}\neq 0$ for instance when $c=d=1$. Injectivity of $\omega\mapsto \omega_\gamma$ and Equation \eqref{comppair} shows that if $\{\omega_l\}_{l=1,...,m}$ is a collection of singular functionals on $\ell^\infty(\N)$ such that $\{\omega_{l}\circ C\}_{l=1,...,m}$ are linearly independent, then the set $\{[c_{\omega_l}]\}_{l=1,...,m}$ will also be linearly independent in $HC^1(C^{1/2}(S^1))$. Therefore, the space in \eqref{spaceofcocy} is infinite-dimensional.
\end{proof}

It is of interest to consider the class of bounded sequences that do not produce a \emph{measurable} operator. Recall that $G\in \mathcal{L}^{1,\infty}(\He)$ is called measurable if $\tra_\omega(G)$ is independent of $\omega$. For $\gamma\in \N_{>1}$, $k\in \N_{>0}$ we set
$$\mathfrak{hms}_{k,\gamma}:=\{c\in \ell^\infty(\N): \; [P,W_{1/2k,\gamma,c}]^{2k}\in \mathcal{L}^{1,\infty}(L^2(S^1)) \quad \mbox{is measurable}\}.$$
Here $\mathfrak{hms}$ stands for $\mathfrak{h}$ankel $\mathfrak{m}$ea$\mathfrak{s}$urable. Proposition \ref{tracexompu} implies that $[P,W_{1/2k,\gamma,c}]^{2k}$ is measurable if and only if $P[P,W_{1/2k,\gamma,c}]^{2k}$ is measurable. We only consider even powers $2k$ because Proposition \ref{tracexompu} implies that $[P,W_{1/(2k+1),\gamma,c}]^{2k+1}$ is measurable for any $\gamma$ and $c$. The following corollary describing $\mathfrak{hms}_{1,\gamma}$ follows from Proposition \ref{hankcomp}.

\begin{cor}
\label{hermajestyscorollary}
The set $\mathfrak{hms}_{1,\gamma}$ does not depend on $\gamma$ and equals
$$\mathfrak{hms}_{1,\gamma}=\left\{c=(c_n)_{n\in \N}\in \ell^\infty(\N): \; \exists d\in \C \;\mbox{ such that }\; \sum_{n=0}^{N} c_n^2=d\cdot N+o(N)\;\mbox{ as }\; N\to \infty\right\}.$$
In particular, the inclusion $\mathfrak{hms}_{1,\gamma}\subseteq \ell^\infty(\N)$ is strict.
\end{cor}

\begin{remark}
What Corollary \ref{hermajestyscorollary} implies is that although $Pa(1-P)aP$ is a nonsmooth pseudo-differential operators, there need not be any classicality in its spectrum.
\end{remark}

In the case that $\sum_{n=0}^{N} c_n^2=d\cdot N+o(N)$ for some $d\in \C$, then $\tra_\omega(P[P,W_{1/2,\gamma,c}]^2)=d/\log(\gamma)$ by Proposition \ref{hankcomp}. To be a bit more specific about Corollary \ref{hermajestyscorollary}, the following sequences are examples making $P[P,W_{1/2,\gamma,c}]^2$ non-measurable:
$$\left(\sqrt{2+\cos(\log(n))}\right)_{n\in \N} \quad \mbox{and} \quad \left(\chi_{ \N\setminus \cup_{k\in \N}[\gamma^{2k},\gamma^{2k+1})}(n)\right)_{n\in \N},$$
where $\chi_{\N\setminus \cup_{k\in \N}[\gamma^{2k},\gamma^{2k+1})}$ denotes the characteristic function of the set $ \N\setminus \cup_{k\in \N}[\gamma^{2k},\gamma^{2k+1})$. The first example above is similar to that in \cite[Corollary 11.5.3]{sukolord}.

\begin{remark}
\label{weirdsequences}
The sequence $ \left(\chi_{ \N\setminus \cup_{k\in \N}[\gamma^{2k},\gamma^{2k+1})}(n)\right)_{n\in \N}$ produces Hankel operators for which the Connes-Dixmier traces depend on the choice of extended limit $\omega=M^*\phi$ (cf. Remark \ref{connesdixremark}). By \cite[Proposition 6, Chapter II.$\beta$]{c}, this assertion is equivalent to the failure of convergence of the following expression as $M\to \infty$:
$$\frac{1}{\log(M)}\sum_{N=1}^M \frac{(-\gamma)^{\lfloor \log_\gamma(\lfloor \log_\gamma(N+1)\rfloor)\rfloor}}{N\cdot \log(N)}.$$
The expression does not converge as subsequent summands of the same sign add a contribution $\Theta(\log(M))$ to the sum, followed by a contribution of opposite sign.
\end{remark}

A constructive consequence of Proposition \ref{hankcomp} is as follows.

\begin{cor}
It holds that
$$\tra_\omega(P[P,W_{1/2,\gamma}]^2)=-\frac{1}{\log(\gamma)}.$$
\end{cor}

\large
\section{Computations on higher-dimensional tori and $\theta$-deformations}
\label{sectiontheta}
\normalsize

The computations in the previous section are based on the structure of the circle; it is a one-dimensional compact Lie group. In this section, we will generalize these computations to higher dimensional manifolds, still in the presence of symmetries, and noncommutative $\theta$-deformations. The noncommutative geometry of noncommutative tori has been well studied, see for instance \cite{bramessuj,c,conneslandi,yamashitadeform}. In recent years, also more analytic aspects have been studied. In \cite{xiongxuyinfirst,xiongxuyin} several tools from classical harmonic analysis are extended to noncommutative tori. It is an interesting open problem to find an approach to harmonic analysis that would work in more general noncommutative geometries, already for general $\theta$-deformations.

Most of the computational tools are present in the general context of higher dimensional manifolds and noncommutative tori, but produce notably more complicated objects. We start by extending some results from Section \ref{introseconriemannain} to $\theta$-deformations in Subsection \ref{thetadefsub}. We compute Dixmier traces in Subsection \ref{thetadefsubcomputing}.

\subsection{Sobolev regularity and $\theta$-deformations}
\label{thetadefsub}

$\theta$-deformation is a standard procedure to deform the algebra of functions on a manifold with a torus action. It preserves the spectrum of equivariant operators, and thereby ensures that the underlying noncommutative geometry, whenever equivariant, deforms well. 

\subsubsection{The noncommutative torus}
We fix an anti-symmetric matrix $\theta\in M_n(\R)$ and identify $\theta$ with a 2-form on $\R^n$. The matrix $\theta=(\theta_{jk})_{j,k=1}^n$ is identified with the 2-cocycle on $\Z^n$ given by 
$$\Z^n\times \Z^n\ni (\mathbbm{k}_1,\mathbbm{k}_2)\mapsto \e^{i\theta(\mathbbm{k}_1,\mathbbm{k}_2)}\in U(1).$$
The $\theta$-deformation of $\T^n$ is denoted by $C(\T^n_\theta)$ and can be defined as the twisted group $C^*$-algebra $C^*(\Z^n, \theta)$. Alternatively, $C(\T^n_\theta)$ can be defined as the universal $C^*$-algebra generated by $n$ unitaries $U_1,\ldots,U_n$ satisfying the commutation relations 
$$U_jU_k=\e^{i\theta_{jk}}U_kU_j, \quad j,k=1,\ldots, n.$$
For $\mathbbm{k}=(k_1,\ldots,k_n)\in \Z^n$, we write $U^{\mathbbm{k}}:=U_1^{k_1}\cdots U_n^{k_n}\in C(\T^n_\theta)$. Clearly, the linear span of $\{U^\mathbbm{k}:\mathbbm{k}\in \Z^n\}$ is dense in $C(\T^n_\theta)$. 

\begin{remark}
The notation $C(\T^n_\theta)$ suggests that $\T^n_\theta$ is an object in its own right -- a ``noncommutative manifold". This point of view is supported by the geometry that survives in a $\theta$-deformation, see \cite{conneslandi,yamashitadeform}. 
\end{remark}

Since $\T^n$ is an abelian Lie group, there is a strongly continuous action of $\T^n$ on $C(\T^n_\theta)$, defined on the generators by 
$$z.U_j:=z_jU_j, \quad z=(z_1,\ldots, z_n)\in \T^n\subseteq \C^n.$$
Moreover, a tracial state $\tau_0\in C(\T^n_\theta)$ is defined by 
$$\tau_0(U^\mathbbm{k}):=\delta_{0,\mathbbm{k}}, \quad \mathbbm{k}\in \Z^n.$$
We denote the GNS-representation of $\tau_0$ by $L^2(\T^n_\theta)$. The set $\{U^\mathbbm{k}:\mathbbm{k}\in \Z^n\}$ is an orthonormal basis for $L^2(\T^n_\theta)$. 

\subsubsection{General $\theta$-deformations} 
Using the noncommutative torus $\T^n_\theta$, we can deform more general $C^*$-algebras. 

\begin{deef}
Let $A$ be a $\T^n-C^*$-algebra, $\He$ a $\T^n$-equivariant Hilbert space and $D$ a $\T^n$-equivariant unbounded closed operator which is densely defined on $\He$.
\begin{itemize}
\item We define the $\theta$-deformation $A_\theta$ as the $\T^n$-invariant $C^*$-subalgebra $(A\otimes C(\T^n_\theta))^{\T^n}\subseteq A\otimes C(\T^n_\theta)$. Here $A\otimes C(\T^n_\theta)$ is equipped with the diagonal $\T^n$-action. The inclusion defines a $*$-monomorphism 
$$i:A_\theta\hookrightarrow A\otimes C(\T^n_\theta)\ ,$$
and we equip $A_\theta$ with the $\T^n$-action defined from $A$ under $i$.
\item We define the $\theta$-deformation $\He_\theta$ as the $\T^n$-invariant subspace $(\He\otimes L^2(\T^n_\theta))^{\T^n}\subseteq \He\otimes L^2(\T^n_\theta)$. Here, the latter space carries the diagonal $\T^n$-action. The inclusion defines an isometry  
$$V:\He_\theta\hookrightarrow \He\otimes L^2(\T^n_\theta).$$
We equip $\He_\theta$ with the $\T^n$-action defined from $\He$ under the isometry $V$. If $\pi:A\to \Bo(\He)$ is a $\T^n$-equivariant representation, we define its $\theta$-deformation as the $\T^n$-equivariant representation 
$$\pi_\theta:A_\theta\to \Bo(\He_\theta),\quad a\mapsto V^*((\pi\otimes \mathrm{id}_{C(\T^n_\theta)})(i(a)))V.$$
\item Let $\tilde{D}$ denote the $\T^n$-equivariant unbounded closed operator $D\otimes \id_{L^2(\T^n_\theta)}$. The $\theta$-deformation of $D$ is defined as the $\T^n$-equivariant unbounded closed operator $V^*\tilde{D}V$.
\end{itemize}
\end{deef} 

The Hilbert spaces we consider are often defined from GNS-representations. Any $\T^n$-invariant trace $\tau\in A^*$ gives rise to a deformed trace $\tau_\theta:=(\tau\otimes \tau_0)\circ i\in A_\theta^*$. From the definition of $\tau_0$, we have the identity 
$$\tau_\theta(a)=\int_{[0,1]^n} (\tau\otimes\alpha_t)(a)\rd t, \; a\in A_\theta\ ,$$ 
where $\alpha_t(a):=\e^{2\pi it}.a$ for $t\in \R^n$. The following proposition follows from a short computation:

\begin{prop}
Let $\tau\in A^*$ be a $\T^n$-invariant tracial state and let $L^2(A_\theta,\tau_\theta)$ denote the GNS-representation of the tracial state $\tau_\theta$. Then the equivariant mapping $i:A_\theta\to A\otimes C(\T^n_\theta)$ induces a unitary isomorphism $L^2(A_\theta,\tau_\theta)\to L^2(A,\tau)_\theta$ compatible with the $A_\theta$-action, i.e. for $\xi_1,\xi_2\in A_\theta$ we have the equality
$$\langle \xi_1,\xi_2\rangle_{L^2(A_\theta,\tau_\theta)}=\langle i(\xi_1),i(\xi_2)\rangle_{L^2(A,\tau)\otimes L^2(\T^n_\theta)}.$$
\end{prop}

The spaces $L^p(A_\theta,\tau_\theta)$ are the noncommutative $L^p$-spaces defined from the tracial state $\tau_\theta$ and satisfy the usual interpolation properties of $L^p$-spaces. We note that $L^\infty(A_\theta,\tau_\theta)$ coincides with $A_\theta''\subseteq \Bo(L^2(A_\theta,\tau_\theta))$. Recall the following well known fact about $\theta$-deformations:

\begin{prop}
\label{deformedequiop}
Let $\He$ be a $\T^n$-equivariant Hilbert space and $D$ a $\T^n$-equivariant unbounded closed operator which is densely defined on $\He$. Define the unitary $\mathcal{U}_\theta:\He\to \He_\theta$ by $\mathcal{U}_\theta\xi:=\xi\otimes U^{-\mathbbm{k}}$ for any $\xi$ which is homogeneous of degree $\mathbbm{k}\in \Z^n$. Then 
$$D_\theta=\mathcal{U}_\theta D \mathcal{U}_\theta^*.$$
In particular, $\theta$-deformations of $\T^n$-equivariant unbounded closed operators are ``isospectral": they preserve the spectra $\sigma(D)=\sigma(D_\theta)$ (including multiplicities).
\end{prop}

\subsubsection{Geometric constructions}
The $\theta$-deformation of the relevant operators allow us to study the Lipschitz functions on a $\theta$-deformed manifold. Let $M$ be a closed Riemannian manifold, and fix a Dirac operator $D$ acting on a Clifford bundle $S\to M$. We consider $D$ as a self-adjoint operator with domain $W^1(M,S)$. A short computation shows that 
$$\Lip(M)=\{a\in C(M):a\Dom(D)\subseteq \Dom(D)\;\mbox{and} \; [D,a]\mbox{ is bounded in $L^2$-operator norm}\}.$$
On the left hand side of this equation, we use Definition \ref{lipdef} (see page \pageref{lipdef}). Henceforth, we assume that $M$ admits a smooth isometric $\T^n$-action that lifts to $S$. We also assume that $D$ is $\T^n$-equivariant. Since the $\T^n$-action is isometric, any Dirac operator on $S$ is a zeroth order perturbation of an equivariant Dirac operator.

\begin{prop}
We define the subspace $\Lip(M_\theta)\subseteq C(M_\theta):=C(M)_\theta$ as 
$$\Lip(M_\theta):=\{a\in C(M_\theta): \,i(a)\in \Lip(M,C(M_\theta))\}.$$
It holds that 
$$\Lip(M_\theta)=\{a\in C(M_\theta):a\Dom(D_\theta)\subseteq \Dom(D_\theta)\;\mbox{and} \; [D_\theta,a]\mbox{ is bounded}\}.$$
\end{prop}

\begin{proof}
It is easily verified that $a\in \Lip(M_\theta)$ if and only if $i(a)$ preserves the domain of $\tilde{D}$ and $[\tilde{D},i(a)]$ extends to a bounded operator on $L^2(M,S)\otimes L^2(\T^n_\theta)$. It follows that $\{a\in C(M_\theta):a\Dom(D_\theta)\subseteq \Dom(D_\theta)\;\mbox{and} \; [D_\theta,a]\mbox{ is bounded}\}\subseteq \Lip(M_\theta)$. The converse follows from the fact that,  locally, the operator $D_\theta$ is implemented by $\sum c(X_i)X_i$ up to lower order terms, where $X_i$ are generators of a local $\R^n$-action on $L^2(M,S)_\theta$. 
\end{proof}

\begin{remark}
\label{thelmapping}
We remark that if $a\in A$ is smooth for the $\T^n$-action, then we can write $a=\sum_{\mathbbm{k}\in \Z^n} a_\mathbbm{k}$, where $a_\mathbbm{k}\in A$ is of degree $\mathbbm{k}$ and defined by 
$$a_\mathbbm{k}:=\int_{[0,1]^n} \alpha_t(a)\e^{-2\pi i\mathbbm{k}\cdot t}\rd t.$$
When $a$ is smooth, the norms $(\|a_\mathbbm{k}\|_A)_{\mathbbm{k}\in \Z^n}$ form a Schwartz sequence. In this case the element $L(a):=\sum_{\mathbbm{k}\in \Z^n} a_\mathbbm{k}\otimes U^{-k}\in A_\theta$ is well defined, and the following identity holds on a core for $D_\theta$:
$$L([D,a])=[D_\theta,L(a)].$$
We remark that this identity is analytically unwieldly when relating the boundedness of the commutators $[D,a]$ and $[D_\theta,L(a)]$; $L$ does not extend to a bounded operator from bounded operators in $\He$ to bounded operators in $\He_\theta$.
\end{remark}

\begin{remark}
The convergence of the ``Fourier series"  $a=\sum_{\mathbbm{k}\in \Z^n} a_\mathbbm{k}$ was studied in detail in \cite{weaverlip}. For $M=\T^n$, it was shown that the Fourier series converges in Cesaro mean, if $a\in C(\T^n_\theta)$ \cite[Theorem 22]{weaverlip}. It converges in norm, if $a\in \Lip(\T^n_\theta)$ \cite[Theorem 23]{weaverlip}. More refined convergence results were obtained in \cite{xiongxuyinfirst}.
\end{remark}

\subsubsection{Littlewood-Paley theory on the noncommutative torus}
The Littlewood-Paley theory on the noncommutative tori uses Fourier theory similarly to the approach on page \pageref{lptheoryontorus} (found in \cite{peller}). The details can be found in \cite{xiongxuyin}. We pick a Littlewood-Paley partition of unity $(\phi_j)_{j\in \N}\subseteq c_c(\Z^n)$ as on page \pageref{lptheoryontorus}, i.e.~$\supp(\phi_j)\subseteq \{\mathbbm{k}\in \Z^n:2^{j-1}\leq |\mathbbm{k}|\leq 2^{j+1}\}$ for $j>0$. We call such a partition of unity a \emph{discrete} Littlewood-Paley partition of unity.

\begin{deef}
We define the Littlewood-Paley decomposition $(\Phi_j^\theta)_{j\in \N}\subseteq \Bo(L^2(\T^n_\theta))$ by 
$$\Phi_j^\theta U^{\mathbbm{k}}:=\phi_j(\mathbbm{k})U^{\mathbbm{k}}.$$
For $s\in \R$ and $p,q\in [1,\infty]$, we consider partially defined norms $\|\cdot\|_{B^s_{p,q}(\T^n_\theta)}$ on $f\in C^\infty(\T^n_\theta)^*$, 
$$\|f\|_{B^s_{p,q}(\T^n_\theta)} :=\|(2^{sj}\|\Phi_j^\theta f\|_{L^p(\T^n_\theta)})_{j\in \N}\|_{\ell^q(\N)}.$$
We define the Banach space $B^s_{p,q}(\T^n_\theta):=\{f\in C^\infty(\T^n_\theta)^*: \|f\|_{B^s_{p,q}(\T^n_\theta)} <\infty\}$.
\end{deef}

In addition, there is a natural Sobolev space scale on $\T^n_\theta$. let $(\delta_j)_{j=1}^n$ denote the generators of the $\T^n$-action. The Laplacian on $\T^n_\theta$ is defined by 
$$\Delta_\theta:=-\sum_{j=1}^n \delta_j^2.$$
For $s\geq 0$ and $p\in [1,\infty)$, we define $W^{s,p}(\T^n_\theta)$ as the domain of $|\Delta_\theta|^{s/2}$ on $L^p(\T^n_\theta)$. Alternatively, for $k\in \N$
$$W^{2k,p}(\T^n_\theta):=\{f\in L^p(\T^n_\theta): \; \Delta_\theta^{k}f\in L^p(\T^n_\theta)\},$$
and one recovers $W^{s,p}(\T^n_\theta)$ by complex interpolation. The space $W^{s,p}(\T^n_\theta)$ for $s<0$ is defined by duality in the $\tau_0$-pairing. We set $W^{s}(\T^n_\theta):=W^{s,2}(\T^n_\theta)$.

\begin{prop}
\label{besovspacecomp}
The Besov spaces satisfy the following properties
\begin{enumerate}
\item For $s\in \R$, $B^s_{2,2}(\T^n_\theta)=W^{s}(\T^n_\theta)$.
\item For any $s\in (0,1)$, 
$$B^s_{\infty,\infty}(\T^n_\theta)=[C(\T^n_\theta),\Lip(\T^n_\theta)]_{s}=C^s(\T^n,C(\T^n_\theta))^{\T^n}.$$
\end{enumerate}
\end{prop}

\begin{proof}
Since $\Delta_\theta U^{\mathbbm{k}}=|\mathbbm{k}|^2U^{\mathbbm{k}}$, the identity $B^s_{2,2}(\T^n_\theta)=W^{s}(\T^n_\theta)$ follows in the classical way from the case $s=0$. To prove the second identity, we note that $[C(\T^n_\theta),\Lip(\T^n_\theta)]_{s}$ equals the invariant subspace of $[C(\T^n,C(\T^n_\theta)),\Lip(\T^n,C(\T^n_\theta))]_{s}$. Standard methods of interpolation show 
\begin{align*}
[C(\T^n,&C(\T^n_\theta)),\Lip(\T^n,C(\T^n_\theta))]_{s}\\
&=\{f\in C(\T^n)\otimes C(\T^n_\theta): \|(2^{sj}\|(\Phi_j^\theta\otimes \id_{C(\T^n_\theta)}) f\|_{L^\infty(\T^n_\theta)})_{j\in \N}\|_{\ell^\infty(\N)} <\infty\}.
\end{align*}
From this identity, the second assertion follows.
\end{proof}

\begin{deef}
For $\alpha\in (0,1)$, we define 
$$C^\alpha(\T^n_\theta):=C^\alpha(\T^n,C(\T^n_\theta))^{\T^n}=[C(\T^n_\theta),\Lip(\T^n_\theta)]_{\alpha}=B^\alpha_{\infty,\infty}(\T^n_\theta).$$ 
\end{deef}

\begin{remark}
It seems possible to extend the Littlewood-Paley theory from $\T^n_\theta$ to $\theta$-deformations of smooth free $\T^n$-actions on manifolds. This is a consequence of the fact that smooth free $\T^n$-actions on manifolds correspond to smooth principal $\T^n$-bundles; therefore we can localize to a product situation. A challenging problem would be to develop Littlewood-Paley theory for arbitrary smooth $\T^n$-actions on manifolds.
\end{remark}

\subsubsection{Sobolev regularity on the noncommutative torus}

We are now ready to prove that certain commutators on the noncommutative torus have Sobolev mapping properties similar to those in Theorem \ref{sobregthm}. The commutators that we are interested in are those between the phase of the Dirac operator and elements of $C^\alpha(\T^n_\theta)$.

On $\T^n$ there is an obvious choice of spin structure coming from the trivialization of the tangent bundle $T\T^n$, as defined by the Lie algebra $\R^n$. Concretely, the spinor bundle is given by a trivial bundle $S=\T^n\times S_n\to \T^n$, and the Dirac operator takes the form $D=\sum_{i=1}^n\gamma_i \partial_{t_i}$. Here $\gamma_i=c(\partial_{t_i})$ and $c:\R^n\to \End(S_n)$ denotes Clifford multiplication. Now note that $S_n=\C^{2^{\lfloor n/2\rfloor}}$ as a vector space; the Clifford structure is determined by the Clifford multiplication $c$. We identify the sections of the spinor bundle with $C^\infty(\T^n)\otimes S_n$. In the Fourier basis $\e_\mathbbm{k}(t):=\e^{2\pi i \mathbbm{k}\cdot t}$, we have 
$$ D(\e_\mathbbm{k}\otimes v)=\e_\mathbbm{k}\otimes c(\mathbbm{k})v, \quad v\in S_n.$$
Define $F:=D|D|^{-1}$. We use the convention that $|D|^{-1}$ acts as $0$ on $\ker(D)$. The operators $D$ and $F$ are $\T^n$-equivariant and induce $\theta$-deformed operators $D_\theta$, respectively $F_\theta=D_\theta|D_\theta|^{-1}$, on $L^2(\T^n_\theta,S)$. The Sobolev spaces satisfy $W^s(\T^n_\theta,S)=\Dom(|D_\theta|^{s})$ for $s\geq 0$, because $|D_\theta|^2=\Delta_\theta\otimes\id_{S_n}$.

\begin{thm}
\label{sobregtheta}
Let $F_\theta$ denote the phase of the $\theta$-deformed Dirac operator $D_\theta$. For any $a\in C^\alpha(\T^n_\theta)$ the operator $[F_\theta,a]$ extends to a bounded operator 
$$[F_\theta,a]:W^s(\T^n_\theta,S)\to W^{s+\alpha}(\T^n_\theta,S),\quad\mbox{for any $s\in (-\alpha,0)$}.$$
\end{thm}

\begin{remark}
Theorem \ref{sobregtheta} holds for arbitrary smooth free $\T^n$-actions on manifolds. This can again be shown by localizing to a product situation. 
\end{remark}

\begin{remark}
Results relating to Theorem \ref{sobregtheta} are the subject of a forthcoming article by Edward McDonald, Fedor Sukochev and Dmitriy Zanin. 
\end{remark}

The proof of Theorem \ref{sobregtheta} is divided into several lemmas. For an element $a\in C(\T^n_\theta)$, we write 
\begin{equation}
\label{aupper}
a^k:=\sum_{j=0}^{k-2} \Phi^\theta_j(a)\ .
\end{equation} 
Associated with the Littlewood-Paley decomposition and $a\in C^\alpha(\T^n_\theta)$, there are operators $T_a:=\sum_{k=0}^\infty a^k\Phi_k^\theta$ and $R_a=\sum_{|j-k|\leq 2} \Phi_j^\theta(a)\Phi_k^\theta$. For $u\in L^2(\T^n_\theta,S)$, we decompose 
$$au=T_au+T_ua+R_au\ .$$ 
We set $\varrho(a)u:=R_au+T_ua$. From these computations, we see that
$$[F_\theta,a]=[F_\theta,T_a]+F_\theta \varrho(a)-\varrho(a)F_\theta.$$
The three terms will be studied separately. The last two terms have the desired mapping property by the next lemma.

\begin{lem}
For $a\in C^\alpha(\T^n_\theta)$ and $s\in (-\alpha,0)$, $\varrho(a)$ extends to a continuous operator
$$\varrho(a):W^s(\T^n_\theta,S)\to W^{s+\alpha}(\T^n_\theta,S)\ .$$
\end{lem}

\begin{proof}
We write $u=\sum_{l=0}^\infty u_l$, where 
\begin{equation}
\label{ulowerl}
u_l:=\Phi_{l}^{\theta}u. 
\end{equation}
By Proposition \ref{besovspacecomp}.1) it suffices to prove that $\varrho(a)$ is continuous as a map $B^s_{2,2}(\T^n_\theta,S)\to B^{s+\alpha}_{2,2}(\T^n_\theta,S)$. When $s+\alpha>0$, we have $\sum_{k=0}^{l+3}2^{2(s+\alpha)k}\sim 2^{2(s+\alpha)l}$. Note that whenever $|j-l|\leq 2$, an argument using the support of the sequence of Fourier coefficients implies
$$\Phi^\theta_k(\Phi_j^\theta(a)\Phi_l^\theta(u))=0, \quad k>l+3\ .$$
For $s+\alpha>0$, we can estimate
\begin{align*}
\|R_au\|_{B^{s+\alpha}_{2,2}}^2&=\sum_{k=0}^\infty 2^{2(s+\alpha)k}\left\|\sum_{|j-l|\leq 2} \Phi_{k}^{\pmb{o}}\left(\Phi_{j}^{a}(a)\cdot u_l\right)\right\|_{L^2}^2\\
&\leq \sum_{|j-l|\leq 2}\sum_{k=0}^{l+3} 2^{2(s+\alpha)k}2^{-2j\alpha}\|a\|_{B^\alpha_{\infty,\infty}}^2\|u_l\|_{L^2}^2+\\
&\lesssim \sum_{|j-l|\leq 2}2^{2sj}\|a\|_{B^\alpha_{\infty,\infty}}^2\|u_l\|_{L^2}^2\lesssim  \|a\|_{B^\alpha_{\infty,\infty}}^2\|u\|_{B^s_{2,2}}^2\ .
\end{align*}
It follows that $R_a$ has the desired mapping properties for $s+\alpha>0$. Following the argument of \cite[Theorem 5.1]{jon95}, one observes that $B^s_{2,2}\times B^\alpha_{\infty,\infty}\ni (u,a)\mapsto T_u a\in B^{s+\alpha}_{2,2}$ is a continuous bilinear mapping for $s<0$.
\end{proof}

It therefore suffices to prove the Sobolev mapping properties for $[F_\theta,T_a]$. Recall the notation $a^k$ from Equation \eqref{aupper} and $u_l$ from Equation \eqref{ulowerl}.

\begin{lem}
\label{alphainbetween}
Theorem \ref{sobregtheta} is true provided there is a constant $C>0$ such that for any $\alpha\in (0,1)$, $a\in C^\alpha(\T^n_\theta)$ and $k\in \N$
$$2^{\alpha k}\|[F_\theta,a^k]\Phi_k^\theta\|_{\Bo(L^2(\T^n_\theta,S))}\leq C\|a\|_{C^\alpha(\T^n_\theta)}\ .$$
\end{lem}

\begin{proof}
Assume that the constant $C>0$ in the statement of the lemma exists. We may then estimate
\begin{align*}
\|[F_\theta,T_a]u\|_{B^{s+\alpha}_{2,2}(\T^n_\theta,S)}^2&=\sum_{l=0}^\infty 2^{2(s+\alpha)k} \left\|\Phi_l^\theta\sum_{k=0}^\infty[F_\theta,a^k]u_k\right\|_{L^2(\T^n_\theta,S)}^2\\
&\lesssim \sum_{|l-k|\leq 2} 2^{2(s+\alpha)k} \|\Phi_l^\theta[F_\theta,a^k]u_k\|_{L^2(\T^n_\theta,S)}^2\\
&\lesssim \sum_{|l-k|\leq 2} 2^{2(s+\alpha)k} \|[F_\theta,a^k]\Phi_k^\theta\|_{\Bo(L^2(\T^n_\theta,S))}^2\|u_k\|_{L^2(\T^n_\theta,S)}^2\\
& \leq C^2\|a\|_{C^\alpha(\T^n_\theta)}^2\|u\|_{B^{s}_{2,2}(\T^n_\theta,S)}^2\ .
\end{align*}
In the last estimate we use the assumption of the lemma.
\end{proof}

\begin{lem}
\label{smoothest}
Theorem \ref{sobregtheta} is true if there is a constant $C>0$ such that for any $a\in C^\infty(\T^n_\theta)$ and $k\in \N$
$$2^{k}\|[F_\theta,a]\Phi_k^\theta\|_{\Bo(L^2(\T^n_\theta,S))}\leq C\|a\|_{\Lip(\T^n_\theta)}\ .$$
\end{lem}

\begin{proof}
By Lemma \ref{alphainbetween}, it suffices to prove that the mapping 
$$C^\alpha(\T^n_\theta)\ni a\mapsto (2^{\alpha k}[F_\theta,a^k]\Phi_k^\theta)_{k\in \N}\in \ell^\infty(\N,\Bo(L^2(\T^n_\theta,S)))\ $$
is well defined and bounded. By interpolation, it suffices to prove that this  mapping is continuous for $\alpha=0$ and $\alpha=1$. At $\alpha=0$ the continuity is clear. At $\alpha=1$, the assumption of the lemma guarantees that 
$$2^{k}\|[F_\theta,a^k]\Phi_k^\theta\|_{\Bo(L^2(\T^n_\theta,S))}\leq C\|a^k\|_{\Lip(\T^n_\theta)}\leq C\|a\|_{\Lip(\T^n_\theta)}\ .$$
\end{proof}

\begin{prop}
\label{lipestim}
Let $M$ be a closed manifold with a smooth $\T^n$-action. If $Q$ is a first order classical $\T^n$-equivariant pseudo-differential operator acting on a $\T^n$-equivariant vector bundle $E$, then there is a constant $C>0$ such that any $a\in \Lip(M_\theta)$ preserves $\Dom(Q_\theta)$ and 
$$\|[Q_\theta,a]\|_{\Bo(L^2(M,E)_\theta)}\leq C\|a\|_{\Lip(M_\theta)}\ .$$
\end{prop}

The proposition follows immediately from the next theorem and Proposition \ref{deformedequiop}.

\begin{thm}
Let $A$ be a $C^*$-algebra represented by $\pi:A\to \Bo(\He)$, $M$ a closed manifold and $Q$ a first order classical pseudo-differential operator on $M$. There exists a $C>0$ such that for any $a\in \Lip(M,A)$, 
$$\|[Q\otimes \id_\He,(\id_{C(M)}\otimes \pi)(a)]\|_{\Bo(L^2(M,\He))}\leq C\|a\|_{\Lip(M,A)}\ .$$
\end{thm}

The theorem is proven by standard techniques in harmonic analysis, for instance using a $T(1)$-theorem. The proof is omitted. We are now ready to prove Theorem \ref{sobregtheta}.

\begin{proof}[Proof of Theorem \ref{sobregtheta}]
By Lemma \ref{smoothest}, it suffices to show that there exists $C>0$ such that for any $a\in C^\infty(\T^n_\theta)$ and $k\in \N$ the estimate $2^{k}\|[F_\theta,a]\Phi_k^\theta\|_{\Bo(L^2(\T^n_\theta,S))}\leq C\|a\|_{\Lip(\T^n_\theta)}$ holds. For $a\in C^\infty(\T^n_\theta)$, we write 
$$[F_\theta,a]=([D_\theta,a]+F_\theta[|D_\theta|,a])|D_\theta|^{-1}\ .$$
By Proposition \ref{lipestim} we can estimate $\|[|D_\theta|,a]\|_{\Bo(L^2(\T^n_\theta,S))}\lesssim \|a\|_{\Lip(\T^n_\theta)}$. We can therefore estimate 
\begin{align*}
\|[F_\theta,a]\Phi_k^\theta\|&_{\Bo(L^2(\T^n_\theta,S))}\leq (\|[D_\theta,a]|D_\theta|^{-1}\Phi_k^\theta\|_{\Bo(L^2(\T^n_\theta,S))}+\|[|D_\theta|,a]|D_\theta|^{-1}\Phi_k^\theta\|_{\Bo(L^2(\T^n_\theta,S))})\\
&\leq (\|[D_\theta,a]\|_{\Bo(L^2(\T^n_\theta,S))}+\|[|D_\theta|,a]\|_{\Bo(L^2(\T^n_\theta,S))})\||D_\theta|^{-1}\Phi_k^\theta\|_{\Bo(L^2(\T^n_\theta,S))}\\
&\lesssim \|a\|_{\Lip(\T^n_\theta)}2^{-k}\ .
\end{align*}
This completes the proof of the theorem.
\end{proof}

\subsection{Computations on higher-dimensional tori}
\label{thetadefsubcomputing}

The ideas in Section \ref{introseconriemannain} and Theorem \ref{sobregtheta} allow us to compute Dixmier traces on the noncommutative torus. The case $\theta=0$ will provide us with higher-dimensional analogues of the computations in Section \ref{hankelsection}. We begin with a corollary of Theorem \ref{vmodthm}, Lemma \ref{lszformnew} and Theorem \ref{sobregtheta}. 

\begin{cor}
\label{thecaomp}
Let $F_\theta$ denote the phase of the deformed Dirac operator, $a_1,\ldots, a_k$ a collection of elements $a_j\in C^{\alpha_j}(\T^n_\theta)$, for $\alpha_j\in (0,1)$ satisfying that $\sum_{j=1}^k \alpha_j=n$, and $T\in \Bo(L^2(\T^n_\theta,S))$. The operator $T[F_\theta,a_1]\cdots[F_\theta,a_k]\in \mathcal{L}^{1,\infty}(L^2(\T^n_\theta,S))$ is $(\alpha_k,s)$-modulated with respect to $D_\theta$ for any $s\in (-\alpha_k,0)$. In particular, if $(v_l)_{l=1}^{2^{\lfloor n/2\rfloor}}$ is an ON-basis for $S_n$ and $\omega\in (\ell^\infty(\N)/c_0(\N))^*$ is an extended limit then 
$$\tra_\omega(T[F_\theta,a_1]\cdots[F_\theta,a_k])=\lim_{N\to \omega}\frac{ \sum_{l=1}^{2^{\lfloor n/2\rfloor}}\sum_{|\mathbbm{k}|\leq N^{1/n}}\langle T[F_\theta,a_1]\cdots[F_\theta,a_k](\e_\mathbbm{k}\otimes v_l),\e_\mathbbm{k}\otimes v_l\rangle}{\log(2+N)}\ .$$
\end{cor}

With this corollary, we now turn to compute Dixmier traces. Fix $F_\theta$ and $a_1,\ldots, a_k$ as in Corollary \ref{thecaomp}, as well as a $\T^n$-equivariant operator $T\in \Bo(L^2(\T^n_\theta,S))$. We let $(a_{j,\mathbbm{k}})_{\mathbbm{k}\in \Z^n}$ and $(T(\mathbbm{k}))_{\mathbbm{k}\in \Z^n}$ denote the Fourier coefficients of $a_j$ and $T$, respectively. The former are defined as in Remark \ref{thelmapping} and the latter as the matrices $T(\mathbbm{k})\in \End(S_n)$ defined by 
$$\langle T(\mathbbm{k})v,w\rangle_{S_n}=\langle T(\e_\mathbbm{k}\otimes v),\e_\mathbbm{k}\otimes w\rangle, \quad v,w\in S_n\ .$$
Any sequence in $\ell^\infty(\Z^n,\End(S_n))$ arises as the Fourier coefficients of a $\T^n$-equivariant operator on $L^2(\T^n_\theta,S)$. We write $\mathbbm{K}=(\mathbbm{k}_1,\ldots, \mathbbm{k}_k)\in (\Z^n)^k$ for a $k$-tuple of Fourier indices in $\Z^n$. For a $k$-tuple $\mathbbm{K}$ of Fourier indices and $\mathbbm{k}_{k+1}\in \Z^n$, we define the matrix 
$$C_{\mathbbm{K},\mathbbm{k}_{k+1}}:=\prod_{j=1}^k\left(\frac{c(\sum_{l=j}^{k+1}\mathbbm{k}_l)}{|\sum_{l=j}^{k+1}\mathbbm{k}_l|}-\frac{c(\sum_{l=j+1}^{k+1}\mathbbm{k}_l)}{|\sum_{l=j+1}^{k+1}\mathbbm{k}_l|}\right)\in \End(S_n)\ .$$
Define $I\subseteq (\Z^n)^k$ as the set of $\mathbbm{K}$ for which $\sum_{j=1}^k\mathbbm{k}_j=0\in \Z^n$. 

\begin{thm}
\label{compuatthm}
Under the assumptions in the preceding paragraph, 
$$\tra_\omega(T[F_\theta,a_1]\cdots[F_\theta,a_k])=\lim_{N\to \omega}
\frac{\sum_{|\mathbbm{k}_{k+1}|\leq N^{1/n}}\sum_{\mathbbm{K}\in I} \left(\prod_{j=1}^k a_{j,\mathbbm{k}_j}\right)\tra_{S_n}\left(T(\mathbbm{k}_{k+1})C_{\mathbbm{K},\mathbbm{k}_{k+1}}\right)}{\log(2+N)}\ .$$
\end{thm}

We note the slightly surprising fact that $\tra_\omega(T[F_\theta,a_1]\cdots[F_\theta,a_k])$ is independent of $\theta$ and only depends on the Fourier coefficients of $T$ and the elements $a_1,\ldots, a_k$.

\begin{proof}
We compute that for $v\in S_n$ and $\mathbbm{k}_{k+1}\in \Z^n$,
\begin{align*}
T[F_\theta,a_1]\cdots&[F_\theta,a_k](\e_{\mathbbm{k}_{k+1}}\otimes v)=\sum_{\mathbbm{K}} \left(\prod_{j=1}^k a_{j,\mathbbm{k}_j}\right) U^{\mathbbm{k}_1}\cdots U^{\mathbbm{k}_k}\e_{\mathbbm{k}_{k+1}}\otimes T\left(\sum_{j=0}^k\mathbbm{k}_j\right)C_{\mathbbm{K},\mathbbm{k}_{k+1}}v\\
&=\sum_{\mathbbm{K}} \left(\prod_{j=1}^k a_{j,\mathbbm{k}_j}\exp\left(i\theta\left(\mathbbm{k}_{j},\sum_{l=j+1}^{k+1} \mathbbm{k}_l\right)\right)\right) \e_{\sum_{j=1}^{k+1}\mathbbm{k}_{j}}\otimes T\left(\sum_{j=1}^{k+1}\mathbbm{k}_j\right)C_{\mathbbm{K},\mathbbm{k}_{k+1}}v\ .
\end{align*}
From this computation and Corollary \ref{thecaomp}, we arrive at the expression
\begin{align*}
\tra_\omega(T&[F_\theta,a_1]\cdots[F_\theta,a_k])\\
&=\lim_{N\to \omega}
\frac{\sum_{|\mathbbm{k}_{k+1}|\leq N^{1/n}}\sum_{\mathbbm{K}\in I} \left(\prod_{j=1}^k a_{j,\mathbbm{k}_j}\e^{i\theta(\mathbbm{k}_{j},\sum_{l=j+1}^{k+1} \mathbbm{k}_l)}\right)\tra_{S_n}\left(T(\mathbbm{k}_{k+1})C_{\mathbbm{K},\mathbbm{k}_{k+1}}\right)}{\log(2+N)}\ .
\end{align*}
It remains to prove that for $\mathbbm{K}\in I$, $\sum_{j=1}^k\sum_{l=j+1}^{k+1}\theta(\mathbbm{k}_{j}, \mathbbm{k}_l)=0$. This identity follows from the following computations:
\begin{align*}
\sum_{j=1}^k\sum_{l=j+1}^{k+1}\theta(\mathbbm{k}_{j},\mathbbm{k}_l)&=\sum_{j=1}^k\sum_{l=j}^{k+1}\theta(\mathbbm{k}_{j}, \mathbbm{k}_l)=\sum_{j=1}^k\sum_{l=j}^{k}\theta(\mathbbm{k}_{j}, \mathbbm{k}_l)\\
&=-\sum_{j=1}^k\sum_{l=j}^{k}\theta(\mathbbm{k}_{l}, \mathbbm{k}_j)=-\sum_{j=1}^k\sum_{l=j}^{k}\theta(\mathbbm{k}_{j}, \mathbbm{k}_l)\ .
\end{align*}
In the first identity, we use $\theta(\mathbbm{k}_j,\mathbbm{k}_j)=0$, in the second identity that $\sum_{j=1}^k\theta(\mathbbm{k}_j,\mathbbm{k}_{k+1})=0$ because $\mathbbm{K}\in I$, in the third identity the antisymmetry of $\theta$ and in the last identity we change the order of summation. We conclude $\sum_{j=1}^k\sum_{l=j+1}^{k+1}\theta(\mathbbm{k}_{j}, \mathbbm{k}_l)=0$ from the fact that $\sum_{j=1}^k\sum_{l=j}^{k}\theta(\mathbbm{k}_{j}, \mathbbm{k}_l)=-\sum_{j=1}^k\sum_{l=j}^{k}\theta(\mathbbm{k}_{j},\mathbbm{k}_l)$.
\end{proof}

\begin{remark}
The cases of interest in Theorem \ref{compuatthm} would be those related to cyclic cocycles as in Proposition \ref{cyclicprop}. In this case, $T=F_\theta$ for $n$ odd and $k$ even and $T=\gamma F_\theta$ for $n$ even and $k$ odd (here $\gamma$ denotes the grading on $S_n$). For $n$ and $k$ simultaneously odd or even there is no obvious candidate for a cyclic cocycle.
\end{remark}

Computations along the lines of Theorem \ref{compuatthm} easily grow  to unmanageable expressions. Traces of products of Clifford matrices can in principle be evaluated, but the dependence on $(\mathbbm{K},\mathbbm{k}_{k+1})$ might be intricate. We illustrate some computations in the case that $n=2$, $k=3$ and $T=\gamma F_\theta$, where $\gamma$ denotes the grading on $S_2$. The spinor space satisfies $S_2\cong \C^2$. We identify $\Z^2=\Z+i\Z\subseteq \C$. Under the isomorphism $S_2\cong \C^2$, 
$$F_\theta(\mathbbm{k})=\begin{pmatrix}0&\frac{\mathbbm{k}}{|\mathbbm{k}|}\\ \frac{\bar{\mathbbm{k}}}{|\mathbbm{k}|}&0\end{pmatrix}\quad\mbox{and}\quad \gamma=\begin{pmatrix}1&0\\ 0&-1\end{pmatrix}\ .$$
We apply the convention that $|0|^{-1}=0$. Note that for $w,z,u\in \T$, the following elementary identity holds:
\begin{equation}
\label{identityforwzu}
w(\bar{w}-\bar{z})(z-u)(\bar{u}-\bar{w})=2i(\mathrm{Im}(u\bar{w})+\mathrm{Im}(w\bar{z}))\ .
\end{equation}
We compute that for a $3$-tuple $\mathbbm{K}\in I\subseteq (\Z^2)^3$ and $\mathbbm{k}_4\in \Z^2$
\small
\begin{align*}
\tra_{S_n}&\left(\gamma C_{\mathbbm{K},\mathbbm{k}_{4}}\right)\\
&=2i\mathrm{Im}\left(\frac{\mathbbm{k}_4}{|\mathbbm{k}_4|}
\left(\frac{\overline{\mathbbm{k}_4}}{|\mathbbm{k}_4|}-\frac{\overline{\mathbbm{k}_2+\mathbbm{k}_3+\mathbbm{k}_4}}{|\mathbbm{k}_2+\mathbbm{k}_3+\mathbbm{k}_4|}\right)
\left(\frac{\mathbbm{k}_2+\mathbbm{k}_3+\mathbbm{k}_4}{|\mathbbm{k}_2+\mathbbm{k}_3+\mathbbm{k}_4|}-\frac{\mathbbm{k}_3+\mathbbm{k}_4}{|\mathbbm{k}_3+\mathbbm{k}_4|}\right)
\left(\frac{\overline{\mathbbm{k}_3+\mathbbm{k}_4}}{|\mathbbm{k}_3+\mathbbm{k}_4|}-\frac{\overline{\mathbbm{k}_4}}{|\mathbbm{k}_4|}\right)\right)\\
&=-4\left(\mathrm{Im}\left(\frac{\mathbbm{k}_3+\mathbbm{k}_4}{|\mathbbm{k}_3+\mathbbm{k}_4|}\frac{\overline{\mathbbm{k}_4}}{|\mathbbm{k}_4|}\right)+\mathrm{Im}\left(\frac{\mathbbm{k}_4}{|\mathbbm{k}_4|}\frac{\overline{\mathbbm{k}_2+\mathbbm{k}_3+\mathbbm{k}_4}}{|\mathbbm{k}_2+\mathbbm{k}_3+\mathbbm{k}_4|}\right)\right)\\
&=-4\left(\frac{\mathbbm{k}_3\times \mathbbm{k}_4}{|\mathbbm{k}_3+\mathbbm{k}_4||\mathbbm{k}_4|}+\frac{\mathbbm{k}_1\times \mathbbm{k}_4}{|\mathbbm{k}_4-\mathbbm{k}_1||\mathbbm{k}_4|}\right)\ .
\end{align*}
\normalsize
Here $\times$ denotes the crossed product defined on $\mathbbm{k}=(k_1,k_2), \mathbbm{k}'=(k'_1,k_2')\in \Z^2$ by $\mathbbm{k}\times \mathbbm{k}'=k_2k_1'-k_1k_2'$.

We consider functions $a_1,a_2,a_3\in C^{2/3}(\T^2_\theta)$. Theorem \ref{compuatthm} implies that for any extended limit $\omega$,
\begin{align*}
\tra_\omega(\gamma F_\theta&[F_\theta,a_1][F_\theta,a_2][F_\theta,a_3])\\
&=\lim_{N\to \omega}
\frac{-4}{\log(2+N)}\sum_{|\mathbbm{k}_{4}|\leq N^{1/2}}\sum_{\mathbbm{K}\in I} \frac{a_{1,\mathbbm{k}_1}a_{2,\mathbbm{k}_2}a_{3,\mathbbm{k}_3}}{|\mathbbm{k}_4|}\left(\frac{\mathbbm{k}_3\times \mathbbm{k}_4}{|\mathbbm{k}_3+\mathbbm{k}_4|}+\frac{\mathbbm{k}_1\times \mathbbm{k}_4}{|\mathbbm{k}_1-\mathbbm{k}_4|}\right)\\
&=\lim_{N\to \omega}
\frac{-4}{\log(2+N)}\sum_{|\mathbbm{k}_{4}|\leq N^{1/2}}\sum_{\mathbbm{k}_1=\mathbbm{k}_2+\mathbbm{k}_3} \frac{a_{1,-\mathbbm{k}_1}a_{2,\mathbbm{k}_2}a_{3,\mathbbm{k}_3}}{|\mathbbm{k}_4|}\left(\frac{\mathbbm{k}_3\times \mathbbm{k}_4}{|\mathbbm{k}_3+\mathbbm{k}_4|}-\frac{\mathbbm{k}_1\times \mathbbm{k}_4}{|\mathbbm{k}_1+\mathbbm{k}_4|}\right)\\
&=\lim_{N\to \omega}
\frac{4}{\log(2+N)}\sum_{|\mathbbm{k}_{4}|\leq N^{1/2}}\sum_{\mathbbm{k}_1=\mathbbm{k}_2+\mathbbm{k}_3} \frac{\mathbbm{k}_1\times \mathbbm{k}_4}{|\mathbbm{k}_4||\mathbbm{k}_4+\mathbbm{k}_1|}(a_{1,-\mathbbm{k}_1}a_{2,\mathbbm{k}_2}a_{3,\mathbbm{k}_3}-a_{1,-\mathbbm{k}_3}a_{2,\mathbbm{k}_2}a_{3,\mathbbm{k}_1})\ .
\end{align*}

\large
\section{Dixmier trace computations on contact manifolds}
\label{sectionformerlyknownas9}
\normalsize

In this section we consider Dixmier traces on higher-dimensional contact manifolds. We recall the geometric setup of contact manifolds in Subsection \ref{contactprelims}. In Subsection \ref{sectionformerlyknownasthesectionformerlyknownas9} we find an integral formula for Dixmier traces, in the style of a Connes type residue trace formula. Subsection \ref{approximatingbyw1d} computes Dixmier traces for operators of the form $T_0[T_1,a_1]\cdots [T_{n+1},a_{n+1}]$ when $a_i$ is Lipschitz in the Carnot-Caratheodory metric: in this case a Connes type residue trace formula holds and the spectral behavior of $T_0[T_1,a_1]\cdots [T_{n+1},a_{n+1}]$ is classical. We first recall relevant facts about the underlying sub-Riemannian geometry.

\subsection{Preliminaries on contact manifolds}
\label{contactprelims}

In this subsection we will introduce notation and provide context for the geometry of contact manifolds. For a more detailed presentation, see for instance \cite{bella96,cadapaaaow} or \cite[Chapter 2]{pongemono}. 

\subsubsection{Heisenberg groups} 
Consider a non-degenerate antisymmetric form $L=(L_{jk})_{j,k=1}^d$ on $\R^d$. A group structure on $\R^{d+1}$, with coordinates $x=(t,z)\in \R\times \R^d= \R^{d+1}$, is obtained from the Lie algebra structure
\[[(t,z),(t',z')]=(L(z,z'),0).\] The corresponding $2$-step nilpotent Lie group $\mathbb{H}_{d+1}$ with product
\[(t,z)\cdot (t',z')=\left(t+t'+\frac{1}{2}L(z,z'),z+z'\right),\]
is called a \emph{Heisenberg group}.
It admits a Lie group action by $\R_+$, 
\begin{equation}
\label{actionr}
\lambda.(t,z)=(\lambda^2t,\lambda z),
\end{equation}
which turns $\mathbb{H}_{d+1}$  into a homogeneous Lie group.

In the coordinates of $\R^{d+1}$, we obtain a $d$-dimensional subbundle $H$ of $T \mathbb{H}_{d+1}$ given by the \emph{horizontal} vector fields
\begin{equation}
\label{leftinvariantframe} X_j:=\frac{\partial}{\partial z_j}+\frac{1}{2}\sum_{k=1}^d L_{jk}z_k\frac{\partial}{\partial t}, \quad j=1,2,\cdots, d.
\end{equation}
The horizontal vector fields are left invariant and homogeneous of degree $1$ with respect to the $\R_+$-action. We let $X_0:=\frac{\partial}{\partial t}$ denote the \emph{vertical} vector field. The vector fields are constructed to satisfy the commutation relation
\[[X_j,X_k]=L_{jk}X_0.\]
Together with the \emph{vertical} vector field, the horizontal vector fields span $T\mathbb{H}_{d+1}$.

\begin{notation}
We keep using the notation $n=d+1$ for the total dimension of a Heisenberg group or a contact manifold. We use $d$ to denote the dimension of the horizontal subbundle $H$ of the tangent bundle.
\end{notation}

The \emph{Koranyi gauge}
$$|x|_H:=\sqrt[4]{x_0^2+\left(\sum_{j=1}^dx_j^2\right)^2},\quad \mbox{for}\quad x=(x_0,x_1,\ldots, ,x_d),$$
defines a natural length function. It satisfies $|\lambda.x|_H=\lambda|x|_H$. For $x\in \mathbb{H}_{d+1}$ and $r>0$ we define the balls
$B_H(x,r):=\{y\in \mathbb{H}_{d+1}:|x^{-1}y|_H<r\}.$ 

\subsubsection{Sub-Riemannian $H$- and contact manifolds} 

A  \emph{Heisenberg structure} on a manifold $M$ is a hyperplane bundle $H\subseteq TM$. We say that $M$ is a \emph{sub-Riemannian $H$-manifold} provided that $H$ is bracket generating, in the sense that $C^\infty(M,H)$ locally generates $C^\infty(M,TM)$ as a Lie algebra. The Lie bracket on $H$ defines a vector bundle morphism $L:H\wedge H \to TM/H$, called the Levi form. We often assume that we have chosen a trivialization of the line bundle $TM/H$. In this case the Levi form becomes a $2$-form on $H$. $L$ endows the tangent space above every point with the structure of $\mathbb{G}:=\R^{d-2m}\times\mathbb{H}_{2m+1}$, if $\dim \ M = d+1$.  

\emph{Contact manifolds} provide a rich class of examples: Here $M$ is of dimension $2n-1$, and $H = \ker \ \eta$, the kernel of a one-form $\eta \in C^\infty(M,T^*M)$ with $\eta\wedge (\rd\eta)^{n-1}$ non-degenerate. Note that $\rd \eta$ is non-degenerate on $H=\ker\eta\subseteq TM$, and recall Cartan's formula,
\[\rd \eta(X,Y)=X(\eta(Y))-Y(\eta(X))-\eta([X,Y])\]
for any vector fields $X,Y$. As $\rd \eta$ is non-degenerate, for any $X\in H$ there is a $Y\in H$ with $\eta([X,Y])\neq 0$. Hence, the Heisenberg structure associated with a contact structure is bracket generating. Locally, at a point $x$, the Heisenberg group structure on the tangent bundle is defined from the Levi form $(\rd \eta)_x$ restricted to $H_x$.

In complex analysis, contact manifolds arise as the boundary $M = \partial \Omega$ of a strictly pseudo-convex domain $\Omega$ in a complex manifold of complex dimension $n$. The 1-form $\eta = \rd^c\rho$ is obtained from a boundary defining function $\rho$.

If $M$ is a sub-Riemannian $H$-manifold, a theorem by Chow assures that for any $x,y\in M$ there is a smooth path $\gamma:[0,1]\to M$ such that $\gamma(0)=x$, $\gamma(1)=y$ and $\dot{\gamma}(t)\in H_{\gamma(t)}$ for almost all $t\in [0,1]$. After a choice of Riemannian metric on $H$, we define the \emph{Carnot-Carath\'eodory metric} as follows:
\small
\begin{equation}
\label{dcccddeef}
\rd_{CC}(x,y):=\inf\left\{\left(\int_0^1\|\dot{\gamma}(t)\|_H^2\rd t\right)^{1/2}:\;\gamma(0)=x,\, \gamma(1)=y\;\mbox{and}\;\dot{\gamma}(t)\in H_{\gamma(t)}\mbox{  a.e.}\right\}.
\end{equation}
\normalsize
In a coordinate chart adapted to a local frame $X_0, X_1, \dots, X_d$, where $X_1, \dots, X_d$ span $H$, the Carnot-Carath\'eodory metric is equivalent to the Koranyi gauge \cite[Lemma 2.6]{gimpgoff}. We define 
$$\Lip_{CC}(M):=\Lip(M,\rd_{CC}).$$

Associated with the sub-Riemannian geometry and a choice of metric on $H$, we obtain geometrically relevant differential operators. We define a sub-Laplacian $\Delta_H$  by the differential expression 
\[\Delta_H:=\rd_H^* \rd_H\ ,\]
where $\rd_H:C^\infty(M)\to C^\infty(M,H^*)$ denotes the exterior differential composed with the fiberwise restriction $C^\infty(M,T^*M)\to C^\infty(M,H^*)$. In a local orthornomal frame $X_1, \dots, X_d$ for $H$, $\Delta_H=-\sum_{j=1}^d X_j^2$ plus lower order terms in $X_1, \dots, X_d$. If $M$ is closed, the closure of the densely defined operator $\Delta_H$ is a self-adjoint non-negative operator with compact resolvent.
 
The horizontal Sobolev spaces on a sub-Riemannian manifold are defined from the sub-Laplacian $\Delta_H$. For $s\geq 0$ and $p\in [1,\infty)$, we define $W^{s,p}_H(M):=(1+\Delta_H)^{-s/2}L^p(M)$. For $s<0$ and $p\in (1,\infty)$, the Sobolev scale is defined using duality in the pairing between $L^p$ and $L^{p'}$, $\frac{1}{p}+\frac{1}{p'}=1$. It is a well known fact that the Sobolev spaces can be localized, i.e. defined from sub-Laplacians in local charts. Moreover, 
\begin{equation}
\label{horizontalestimaate}
W^{1,p}_H(M)=\{f\in L^p(M): Xf\in L^p(M)\;\forall X\in C^\infty(M,H)\}.
\end{equation}

\subsubsection{Heisenberg pseudo-differential operators}
\label{Heisenbergcalc}

We briefly review a pseudo-differential calculus adapted to a Heisenberg structure, referring to \cite{bg,pongemono} for a more detailed account. Let $M$ be a $d+1$-dimensional closed sub-Riemannian $H$-manifold and $U$ a local coordinate chart adapted to a local frame $X_0,\ldots, X_d$ for $H$. Pseudo-differential operators arise as quantizations of symbols from the following class, see \cite[Definition 3.1.4 and 3.1.5]{pongemono}.

\begin{deef}
Let $m\in \C$. The symbol space $S^m(U\times \R^{d+1})$ is defined as the space of all $p\in C^\infty(U\times \R^{d+1})$ that admit a polyhomogeneous asymptotic expansion of order $m$: For all $k\in \N$ there exists $p_k\in C^\infty(U\times \R^{d+1}\setminus \{0\})$, with $p_k(x,\lambda\cdot \xi)=\lambda^{m-k}p_k(x,\xi)$ when $\lambda>0$, such that for all $N\in \N$ and compact $K\subseteq U$ 
we have for some $C_{\alpha,\beta,K,N}>0$: \[\left| \partial_x^\alpha\partial_\xi^\beta\left(p-\sum_{k=0}^Np_k\right)(x,\xi)\right|\leq C_{\alpha,\beta,K,N}|\xi|_H^{\Re(m)-\langle \beta\rangle-N},\quad\forall \alpha, \beta \in \N^{d+1}, \;x\in K, \;|\xi|_H\geq 1.\] 
Here $\langle \beta\rangle:=2\beta_0+\sum _{i=1}^d \beta_i$ for $\beta=(\beta_0,\beta_1,\ldots,\beta_d)\in \N^{d+1}$.
\end{deef}

We denote the classical symbols of the differential operators $X_0,X_1,\ldots, X_d$ by $\sigma_j(x,\xi):=\sigma(X_j)$, and let $\sigma = (\sigma_0,\sigma_1,\ldots,\sigma_d)$.  Using the formula
\[Pf(x):=(2\pi)^{-d-1}\int \e^{ix\cdot \xi}p(x,\sigma(x,\xi))\hat{f}(\xi)\rd \xi\ ,\] a symbol $p\in S^m(U\times \R^{d+1})$ induces an operator $P:=p(x,-iX):C^\infty_c(U)\to C^\infty(U)$. We say that $P\in \Psi_H^{m}(U)$, or $P$ is a $\Psi_H$DO, provided $P=p(x,-iX)+R$ for $p\in S^m(U\times \R^{d+1})$ and  $R$ an integral operator with smooth integral kernel. 

By \cite[Proposition 3.1.18]{pongemono}, $\Psi_H^{m}(U)$ is invariant under changes of Heisenberg charts. This allows to define the space $\Psi_H^m(M)$ of Heisenberg pseudo-differential operators of order $m$ on $M$. A fundamental theorem is:

\begin{thm}
\label{l2andwhatnot}
If $P\in \Psi^0_H(M)$, then $P$ extends to a bounded operator on $L^2(M)$.
\end{thm}

\begin{remark}
\label{PsiDOprop}
There is a kernel characterization of operators $P\in \Psi^m_H(M)$ by \cite[Proposition 3.1.16]{bg}. The Schwartz kernel $k_P$ of an operator $P \in  \Psi_H^{m}(M)$ satisfies estimates similar to a Calder\'{o}n-Zygmund kernel. In a local frame, $k_P$ can be written as
\[k_P(x,y)=|\epsilon'_x|K_P(x,-\epsilon_x(y))+R(x,y),\]
where $\epsilon_x$ denotes privileged coordinates depending smoothly on $x$, $R$ is smoothing and $K_P$ admits a homogeneous expansion as in \cite[Definition 3.1.11 and 3.1.13]{bg}. The Schwartz kernel $k_P$ is smooth away from the diagonal and satisfies
\begin{align*}
|k_P(x,-\epsilon_x(y))| &\lesssim \rd_{CC}(x,y)^{-(d+1+m)}\\
|V_{x}V_y k_P(x,-\epsilon_x(y))| &\lesssim \rd_{CC}(x,y)^{-(d+1+m+k)},
\end{align*}
for $x \neq y$ and any products $V_x$ and $V_y$ of horizontal differential operators acting on $x$ and $y$, respectively, with total order $k$.
\end{remark}

\subsection{A Connes type formula on the Hardy space}
\label{sectionformerlyknownasthesectionformerlyknownas9}

We will now turn to a computation of Dixmier traces of weak trace class operators acting on the image of a Szeg\"o projection. We focus on the $2m-1$-dimensional sphere which is given the contact structure coming from its realization as the boundary of the open unit ball in $\C^m$. Any contact manifold is locally modeled on a sphere, by Darboux' theorem. The Szeg\"o projection $P\in \Psi^0_H(S^{2m-1})$ on the $2m-1$-dimensional sphere is given by 
$$Pf(z):=\frac{1}{(2\pi i)^m}\int_{S^{2m-1}} \frac{f(w)\rd S(w)}{(1-z\cdot\bar{w})^m}.$$
The right hand side should be interpreted as an interior limit.

\begin{thm}
\label{contform}
Assume that $G\in \mathcal{L}^{1,\infty}(H^2(S^{2m-1}))$ is a weakly sub-Laplacian modulated operator with kernel $k_G$. Then 
$$\tra_\omega(G)=\lim_{N\to \omega} \frac{1}{\log(N)}\int_{S^{2m-1}\times S^{2m-1}} k_G(z,w)\,h_N(1-z\cdot \bar{w})\rd V(z) \rd V(w),$$
where
$$h_N(t):=\frac{1}{n!} \frac{\rd^{m-1}}{\rd t^{m-1}} \frac{1-(1-t)^{N+m}}{t}.$$
\end{thm}

\begin{remark}
Let $\theta$ denote the contact form on $S^{2m-1}$. If $\exp:TS^{2m-1}\to S^{2m-1}\times S^{2m-1}$, we note that $|1-z\bar{w}-(|v|^2+i\theta(v))|=O(\rd_{CC}(z,w))$ and $\left||v|^2+i\theta(v)\right|\sim \rd_{CC}(x,y)^2$ whenever $\exp(x,v)=(z,w)$ is close enough to the diagonal. These facts raise the question of whether Theorem \ref{contform} holds true for localizable sub-Laplacian modulated operators on a Hardy space  for a general contact manifold when replacing $1-z\bar{w}$ with $|v|^2+i\theta(v)$ (for the contact form $\theta$). We note that under these assumptions on $G$, a formula similar to that in Theorem \ref{contform} can be written up on a general contact manifold once a covering of contact coordinate charts has been made. 
\end{remark}

\begin{proof}[Proof of Theorem \ref{contform}]
Define an ON-basis for $H^2(S^{2m-1})$ by $e_\alpha(z)=c_\alpha z^\alpha$ for $\alpha\in\N^m$ where 
$$c_\alpha:=\sqrt{\frac{(m+|\alpha|-1)!}{(m-1)!\alpha!}}.$$
We consider the Kohn sub-Laplacian 
$$\Delta_K:=\sum_{1\leq j<k\leq m} M_{jk}\bar{M}_{jk}+\bar{M}_{jk}M_{jk}, \quad\mbox{where}\quad M_{jk}:=\bar{z}_j\frac{\partial}{\partial z_k}-\bar{z}_k\frac{\partial}{\partial z_j}.$$
The operator $\Delta_K$ is $H$-elliptic of order $2$ and a computation shows that $\Delta_K e_\alpha=(m-1)|\alpha|e_\alpha$. We can by Theorem \ref{vmodthm} write 
\begin{align}
\nonumber
\tra_\omega(G)&=\lim_{N\to \omega} \frac{1}{m\log(N)}\sum_{|\alpha|\leq N} \langle Ge_\alpha,e_\alpha\rangle\\
\label{bascomput}
&=\lim_{N\to \omega} \frac{1}{m\log(N)}\int_{S^{2m-1}\times S^{2m-1}} \left[\sum_{|\alpha|\leq N} c_\alpha^2 z^\alpha \bar{w}^\alpha\right] k_G(z,w)\rd V(z) \rd V(w).
\end{align}
The binomial theorem implies the identity 
\begin{align*}
\sum_{|\alpha|\leq N} \frac{(m+|\alpha|-1)!}{(m-1)!\alpha!} z^\alpha \bar{w}^\alpha&=\sum_{k=0}^N\frac{(m+k-1)(m+k-2)\cdots (k+1)}{(m-1)!}(z\bar{w})^k\ .
\end{align*}
A geometric series computation shows 
\begin{align*}
\sum_{k=0}^N\frac{(m+k-1)(m+k-2)\cdots (k+1)}{(m-1)!}t^k&=\frac{1}{(m-1)!} \frac{\rd^{m-1}}{\rd t^{m-1}} \frac{1-t^{N+m}}{1-t}=mh_N(1-t).
\end{align*}
These computations show $\sum_{|\alpha|\leq N} c_\alpha^2 z^\alpha \bar{w}^\alpha=mh_N(1-z\bar{w})$ and the lemma follows from Equation \eqref{bascomput}.
\end{proof}

\begin{remark}
It follows from the proof of Theorem \ref{contform} that the horizontal Sobolev scale on $S^{2m-1}$ satisfies that
$$W^s_H(S^{2m-1})\cap H^2(S^{2m-1})=\left\{f=\sum_{\alpha\in \N^m} a_\alpha e_\alpha: \;(|\alpha|^{s/2}a_\alpha)_{\alpha\in \N^m}\in \ell^2(\N^m)\right\}.$$
Here $e_\alpha(z):=\frac{z^\alpha}{\|z^\alpha\|_{L^2(S^{2m-1})}}$. We note that an elementary computation with partial derivatives shows the analogous equality in the interior of $S^{2m-1}$
$$W^s(B_{2m})\cap \mathcal{O}(B_{2m})=\left\{f=\sum_{\alpha\in \N^m} a_\alpha \tilde{e}_\alpha: \;(|\alpha|^{s}a_\alpha)_{\alpha\in \N^m}\in \ell^2(\N^m)\right\},$$
where $\tilde{e}_\alpha(z):=\frac{z^\alpha}{\|z^\alpha\|_{L^2(B_{2m})}}$. By using polar coordinates, one sees that 
$$\tilde{e}_\alpha|_{S^{2m-1}}=(2|\alpha|+2m)^{1/2}e_\alpha.$$
Therefore, the trace mapping induces a unitary isomorphism 
$$W^s(B_{2m})\cap \mathcal{O}(B_{2m})\to W^{2s-1}_H(S^{2m-1})\cap H^2(S^{2m-1})\quad\mbox{for all $s\geq \frac{1}{2}$}.$$
In general, if $\Omega$ is a relatively compact strictly pseudo-convex domain in an $m$-dimensional complex manifold, $\overline{\Omega}$ is near the boundary locally biholomorphically equivalent to a neighborhood of a point on $S^{2m-1}$ in $\overline{B_{2m}}$. From the localizability of Sobolev spaces we deduce the following proposition.
\end{remark}

\begin{prop}
Assume that $\Omega$ is a relatively compact strictly pseudo-convex domain in a complex manifold of complex dimension $m$. The trace mapping induces an isomorphism 
$$W^s(\Omega)\cap \mathcal{O}(\Omega)\to W^{2s-1}_H(\partial \Omega)\cap H^2(\partial \Omega)\quad\mbox{for all $s\geq \frac{1}{2}$}.$$
\end{prop}

\subsection{Approximation of Lipschitz functions and Wodzicki residues}
\label{approximatingbyw1d}

In this subsection we study Dixmier traces of operators of the form $T_0[T_1,f_1]\cdots [T_{k},f_{k}]$ when the spectral asymptotics is not governed by singularities. Here $T_0$ is a bounded operator and $T_1,\ldots, T_k\in \Psi^0_H(M)$. By showing that the Dixmier traces are continuous in slightly weaker norms than the H\"older norm, we extend formulas for Dixmier traces from smooth functions by continuity to larger spaces. In the particular case when $f_1,\dots,f_k$ are Lipschitz, the smooth functions are dense in $\Lip_{CC}(M)$ with respect to these weaker norms, and Dixmier traces can be computed using Ponge's Wodzicki residue in the Heisenberg calculus \cite{pongeresidue}, analogously to Theorem \ref{ctfforlip}.

The basic idea in this subsection is to use techniques of Rochberg-Semmes \cite{rochbergsemmes} to estimate singular values, and we generalize these to the sub-Riemannian setting. The work of Feldman-Rochberg \cite{feldmanrochberg} extended the results of \cite{rochbergsemmes} to the Szeg\"o projection on the unit sphere -- the results in this subsection goes even further. 

To rephrase the results of Rochberg-Semmes for sub-Riemannian $H$-manifolds, we need some further notation. Recall from Section \ref{contactprelims} that a sub-Riemannian $H$-manifold is locally modeled on $\mathbbm{G}=\R^{d-2m}\times \mathbbm{H}_{2m+1}$. We let $\Gamma=\Z^{d-2m}\times \Gamma_{2m+1}$ denote the standard lattice in $\mathbbm{G}$, i.e. $\Gamma$ is identified with $\Z^{d+1}$ under a suitable identification of $\mathbbm{G}$ with $\R^{d+1}$. The lattice $\Gamma_{2m+1}$ is denoted by $\Gamma_{\bf 1}=\Gamma_{(1,1,\ldots,1)}$ in \cite{follandcr}. The anisotropic scalings by $\lambda>0$ (from Equation \eqref{actionr}) will be denoted by $\delta_\lambda:\mathbbm{G}\to \mathbbm{G}$. 

We decompose $\mathbbm{G}$ using Christ cubes, see \cite{christcubepap}. We follow the construction in \cite[Section 2.C1]{meyersonpaper}; there the construction is carried out for the Heisenberg group but the generalization to $\mathbbm{G}=\R^{d-2m}\times \mathbbm{H}_{2m+1}$ is straightforward. The technical details are found in \cite[Section 3]{christcubepap}. Let $\mathcal{Q}_C$ denote a Christ cube centered at $0\in \mathbbm{G}$, written $Q(0,\alpha)$ in the notation of \cite[Section 2.C1]{meyersonpaper}. The set $\mathcal{Q}_C$ is open and pre-compact. For a suitable re-scaling, that we suppress for notational simplicity, the family $\{\gamma \mathcal{Q}_C: \gamma \in \Gamma\}$ partitions  $\mathbbm{G}$ up to a set of measure $0$. 

For $\gamma\in \Gamma$ and $\beta\in \Z$ we define
$$\mathcal{Q}_{\gamma,\beta}:=\delta_{2^{\beta}}\left(\gamma.\mathcal{Q}_{C}\right)\subseteq \mathbbm{G}.$$
For any $\beta$, $\mathbbm{G}\setminus \cup_{\gamma\in \Gamma}\mathcal{Q}_{\gamma,\beta}$ has measure zero. We write $\mathcal{I}=\{\mathcal{Q}_{\gamma,\beta}: \,\gamma\in \Gamma, \,\beta\in \Z\}$. We often identify $\mathcal{I}$ with $\Gamma\times \Z$, and sometimes with the subset $\{(\delta_{2^\beta} \gamma,2^\beta): \,\gamma\in \Gamma, \,\beta\in \Z\}\subseteq \mathbbm{G}\times \R_+$. The set $\mathcal{I}$ has the property that if $\mathcal{Q}_1,\mathcal{Q}_2\in \mathcal{I}$ satisfy $\mathcal{Q}_1\cap \mathcal{Q}_2\neq \emptyset$, then $\mathcal{Q}_1\subseteq \mathcal{Q}_2$ or $\mathcal{Q}_2\subseteq \mathcal{Q}_1$. Moreover, for any $(\gamma_1,\beta_1)$ and $\beta_2\geq \beta_1$ there is a unique $\gamma_2$ such that $\mathcal{Q}_{\gamma_1,\beta_1}\subseteq \mathcal{Q}_{\gamma_2,\beta_2}$. For $\mathcal{Q}=\mathcal{Q}_{\gamma,\beta}\in \mathcal{I}$ we write
$$\xi(\mathcal{Q}):=\delta_{2^\beta} \gamma\in \mathbbm{G}\quad\mbox{and}\quad \eta(\mathcal{Q}):=2^\beta\in \R_+.$$
We also write $|\mathcal{Q}|$ for the euclidean volume of $\mathcal{Q}$. Note that $|\mathcal{Q}|\sim \eta(\mathcal{Q})^{d+2}$. Motivated by Rochberg-Semmes, we make the following definition.

\begin{deef}
A sequence $(e_\mathcal{Q})_{\mathcal{Q}\in \mathcal{I}}\subseteq L^2(\mathbbm{G})$ is called an HNWO-sequence (Heisenberg Nearly Weakly Orthogonal) if there is a $C=C((e_\mathcal{Q})_\mathcal{Q})>0$ such that for any $f\in L^2(\mathbbm{G})$, we have 
$$\|f^*\|_{L^2}\leq C\|f\|_{L^2},$$
where 
$$f^*(x):=\sup\left\{ \frac{|\langle f,e_\mathcal{Q}\rangle|}{|\mathcal{Q}|^{1/2}}: \mathcal{Q}\in \mathcal{I} \mbox{  s.t.  }|x^{-1}\xi(\mathcal{Q})|_H<\eta(\mathcal{Q})\right\}.$$
We call $C$ the HNWO-constant of $(e_\mathcal{Q})_\mathcal{Q}$.
\end{deef}

The analogous definition in the euclidean case was called an NWO-sequence in \cite{rochbergsemmes}. The reader is referred to the discussion in \cite{rochbergsemmes} regarding the uses of NWO-sequences. In \cite{feldmanrochberg}, the holomorphic extension of the kernel of the Szeg\"o projection gave rise to the HNWO-sequence needed to prove estimates of singular values.

\begin{prop}[cf. \cite{rochbergsemmes}, bottom of page 239]
\label{lpboundsandhnwo}
Let $(e_\mathcal{Q})_{\mathcal{Q}\in \mathcal{I}}\subseteq L^2(\mathbbm{G})$ be a sequence of functions such that $\supp e_{\mathcal{Q}}\subseteq \delta_3(\mathcal{Q})$ for each $\mathcal{Q}$ and such that there is a $p>2$ with $\|e_\mathcal{Q}\|_{L^p}\leq c|\mathcal{Q}|^{\frac{1}{p}-\frac{1}{2}}$ for each $\mathcal{Q}$. Then $(e_\mathcal{Q})_{\mathcal{Q}\in \mathcal{I}}\subseteq L^2(\mathbbm{G})$ is an HNWO-sequence.
\end{prop}

\begin{proof}
For $\mathcal{Q}\in \mathcal{I}$ and $h\in L^2(\mathbbm{G})$, we estimate for $|x^{-1}\xi(\mathcal{Q})|_H<\eta(\mathcal{Q})$
\begin{align*}
|\mathcal{Q}|^{-1/2} |\langle h, e_\mathcal{Q}\rangle|&\leq c|\mathcal{Q}|^{\frac{1}{p}-1}\|h \chi_{\delta_3(\mathcal{Q})}\|_{L^{p/(p-1)}}\\
&\leq c |\mathcal{Q}|^{-\frac{p-1}{p}}\left(\int_{\delta_3(\mathcal{Q})} |h|^{\frac{p}{p-1}}\right)^{(p-1)/p}\leq c' M(|h|^{\frac{p}{p-1}})^{\frac{p-1}{p}},
\end{align*}
where $M$ is the maximal function defined by 
$$Mf(x):=\sup\left\{ |\mathcal{Q}|^{-1} \int_{\mathcal{Q}} |f|: |x^{-1}\xi(\mathcal{Q})|_H<\eta(\mathcal{Q})\right\}.$$
It follows from \cite[Theorem 2.4.b]{follandstein} (or a classical ball counting argument) that $M$ is bounded on $L^q$ for $q>1$. Hence $\|h^*\|_{L^2}\leq c'\|M\|_{L^{2-\frac{2}{p}}\circlearrowleft}\|h\|_{L^2}$, and $(e_\mathcal{Q})_{\mathcal{Q}\in \mathcal{I}}$ is an HNWO-sequence.
\end{proof}

\begin{lem}[\cite{rochbergsemmes}]
Let $K$ be an integral operator on $L^2(\mathbbm{G})$ with kernel $k=k(x,y)$ and $p,q\in [1,\infty]$.
\begin{enumerate}
\item Suppose that there are two HNWO-sequences $(e_\mathcal{Q})_{\mathcal{Q}\in \mathcal{I}}$ and $(f_\mathcal{Q})_{\mathcal{Q}\in \mathcal{I}}$ such that 
\begin{equation}
\label{decomposinghnwo}
k(x,y)=\sum_{\mathcal{Q}\in \mathcal{I}} \lambda_\mathcal{Q} e_\mathcal{Q}(x)f_\mathcal{Q}(y).
\end{equation}
Then $\|K\|_{\mathcal{L}^{p,q}}\leq C\|(\lambda_\mathcal{Q})_{\mathcal{Q}}\|_{\ell^{p,q}(\mathcal{I})}$, where $C$ depends only on the HNWO-sequences $\{e_\mathcal{Q}\}_{\mathcal{Q}\in \mathcal{I}}$ and $\{f_\mathcal{Q}\}_{\mathcal{Q}\in \mathcal{I}}$.
\item Suppose that $k=\sum_{\mathbbm{k}\in \Z^N} k_{\mathbbm{k}}$, where each $k_{\mathbbm{k}}$ decomposes as in Equation \eqref{decomposinghnwo} for sequences $(\lambda_{\mathcal{Q},\mathbbm{k}})_{\mathcal{Q}\in \mathcal{I}}$ and two HNWO-sequences $(e_{\mathcal{Q},\mathbbm{k}})_{\mathcal{Q}}$ and $(f_{\mathcal{Q},\mathbbm{k}})_{\mathcal{Q}}$ with uniformly bounded HNWO-constants. Then there is a constant $C>0$ only depending on the HNWO-sequences such that for a large enough $a>0$, 
$$\|K\|_{\mathcal{L}^{p,q}}\leq C\sup_{\mathbbm{k}} |\mathbbm{k}|^a\|(\lambda_{\mathcal{Q},\mathbbm{k}})_\mathcal{Q})\|_{\ell^{p,q}(\mathcal{I})}.$$
\end{enumerate}
\end{lem}

The proof of part 1) of the lemma can be found on \cite[page 241]{rochbergsemmes} and part 2) follows from part 1) and \cite[Lemma 1.16]{rochbergsemmes}.

\begin{lem}[cf. Proposition 4.1 of \cite{rochbergsemmes}]
\label{threeeq}
Let $f\in L^1_{loc}(\mathbbm{G})$, $p,q\in [1,\infty]$, $\beta>0$ and assume that $\{e_\mathcal{Q}\}_{\mathcal{Q}\in \mathcal{I}}$ is an HNWO-sequence. For $\mathcal{Q}\in \mathcal{I}$, we define 
$$f_\mathcal{Q}:=\frac{1}{|\mathcal{Q}|}\int_\mathcal{Q}f(y)\rd y.$$
For $r\geq 1$, we set
$$\mathrm{Osc}(f,r,\mathcal{Q}):=\left[\frac{1}{|\mathcal{Q}|}\int_\mathcal{Q}|f(x)-f_\mathcal{Q}|^r\rd x\right]^{1/r}.$$
The following statements are equivalent:
\begin{enumerate}
\item The sequence $(\mathrm{Osc}(f,r,\mathcal{Q}))_{\mathcal{Q}}$ belongs to $\ell^{p,q}(\mathcal{I})$ for all $r\geq 1$.
\item The sequence $(\mathrm{Osc}(f,r,\mathcal{Q}))_{\mathcal{Q}}$ belongs to $\ell^{p,q}(\mathcal{I})$ for $r= 1$.
\item There exists a function $F\in C^1(\mathbbm{G}\times \R_+)$ with $\lim_{s\to 0} F(x,s)=f(x)$ and $F^*\in L^{p,q}(\mathbbm{G}\times \R_+, s^{-d-3}\rd s\rd x)$ where
$$F^*(x,s):= \sup\left\{v|\tilde{\nabla}_H F(u,v)|: \; |ux^{-1}|_H<\beta s, v\in \left(s/2,s\right)\right\}.$$
Here $\tilde{\nabla}_H$ denotes the horizontal gradient on $\mathbbm{G}\times \R_+$.
\end{enumerate}
In particular, for any $r\geq 1$, there exists a constant $C>0$ such that 
$$\|(\mathrm{Osc}(f,r,\mathcal{Q}))_{\mathcal{Q}}\|_{\ell^{p,q}(\mathcal{I})}\leq C \inf\{\|F^*\|_{L^{p,q}(\mathbbm{G}\times \R_+, s^{-d-3}\rd s\rd x)}: \;F\in C^1(\mathbbm{G}\times \R_+)\mbox{   with   } F|_{s=0}=f\}.$$
\end{lem}

The proof proceeds identically to the one of \cite[Proposition 4.1]{rochbergsemmes}, apart from the fact that $|\mathcal{Q}|\sim \eta(\mathcal{Q})^{d+2}$, which explains the different weight appearing in the $L^{p,q}$-space of Condition (3). Motivated by Lemma \ref{threeeq}, we drop $r$ from the notation and write simply $\mathrm{Osc}(f,\mathcal{Q})$.

\begin{deef}
We define
$$\mathrm{Osc}^{p,q}_H(\mathbbm{G}):=\{f\in L^1_{loc}(\mathbbm{G}): (\mathrm{Osc}(f,\mathcal{Q}))_{\mathcal{Q}\in \mathcal{I}}\in \ell^{p,q}\}.$$
If $M$ is a closed sub-Riemannian $H$-manifold, we write 
$$\mathrm{Osc}^{p,q}_H(M):=\{f\in L^1_{loc}(M): f \mbox{    belongs to $\mathrm{Osc}^{p,q}_H$ in each coordinate chart}\}.$$
\end{deef}

It is immediate from the construction that $\mathrm{Osc}^{p,q}_H(\mathbbm{G})$ and $\mathrm{Osc}^{p,q}_H(M)$ are Banach spaces in the norm defined from the embedding into $\ell^{p,q}(\mathcal{I})$. We now come to the main technical result of this subsection.

\begin{thm}
\label{oscestimate}
If $M$ is a closed sub-Riemannian $H$-manifold and $Q\in \Psi^0_H(M)$ there is a constant $C_Q>0$ such that for $a\in \mathrm{Osc}^{p,q}_H(M)$ we can estimate
$$[Q,a]\in \mathcal{L}^{p,q}(L^2(M))\quad\mbox{and}\quad \|[Q,a]\|_{\mathcal{L}^{p,q}}\leq C_Q \|a\|_{\mathrm{Osc}^{p,q}_H(M)}.$$
\end{thm}

For the proof, one argues mutatis mutandis as in \cite[Chapter II.A]{rochbergsemmes}. We only recall its outline. The main difference is that the Calderon-Zygmund kernel $k_Q$ of $Q$ satisfies different estimates (see Remark \ref{PsiDOprop}) than in the Riemannian case (see \cite[Equation (2.2)]{rochbergsemmes}). This does not alter the structure of the proof, since the diagonal growth behavior is dampened by the size of the cubes in \cite[Lemma 2.9]{rochbergsemmes}, as captured in the next lemma.

We form a Whitney decomposition $\mathcal{P}$ of $\mathbbm{G}\times \mathbbm{G}\setminus \textnormal{Diag}$ using Heisenberg dyadic cubes. Here $\textnormal{Diag}:=\{(x,x): x\in \mathbbm{G}\}$ denotes the diagonal. The construction of the Whitney decomposition goes as in \cite[Chapter VI.1]{steindiff} using the dyadic decomposition of $\mathbbm{G}\times \mathbbm{G}$ defined from the lattice $\Gamma\times \Gamma\subseteq \mathbbm{G}\times \mathbbm{G}$ and its Christ cube $\mathcal{Q}_C\times \mathcal{Q}_C\subseteq  \mathbbm{G}\times \mathbbm{G}$. Its main properties are that a cube $\mathcal{Q}=\mathcal{Q}_1\times\mathcal{Q}_2\in \mathcal{P}$ satisfies that $\mathcal{Q}_1,\mathcal{Q}_2\in \mathcal{I}$ and 
$$\mathrm{diam}(\mathcal{Q}_1)= \mathrm{diam}(\mathcal{Q}_2)\sim \eta(\mathcal{Q}_i)\sim \mathrm{dist}_{\rd_H}(\mathcal{Q},\textnormal{Diag}), \quad i=1,2.$$
Moreover, we can guarantee that there is a number $N_0$ such that for a point $x\in \mathbbm{G}\times \mathbbm{G}\setminus \textnormal{Diag}$ there are at most $N_0$ cubes $\mathcal{Q}\in \mathcal{P}$ with $x\in \frac{6}{5}\mathcal{Q}$. 

The characteristic function of a cube $\mathcal{Q}$ will be denoted by $\chi_\mathcal{Q}$. Let $k$ denote the integral kernel of $Q$ and decompose $k=\sum_{\mathcal{Q}\in \mathcal{P}} k_\mathcal{Q}$, where $k_\mathcal{Q}=k\chi_\mathcal{Q}$.

\begin{lem}
Let $k$ denote the Calderon-Zygmund kernel of an operator in $\Psi^0_H(M)$. There exists an $a_0>0$ such that for $a>0$ there is a constant $C=C(k)>0$ such that for any cube $\mathcal{Q}=\mathcal{Q}_1\times \mathcal{Q}_2\in \mathcal{P}$ there exists:
\begin{enumerate}
\item a sequence of numbers $(\lambda_{\mathcal{Q},\mathbbm{k}})_{\mathbbm{k}\in \Z^{2(d+1)}}$ for which $|\mathbbm{k}|^a |\lambda_{\mathcal{Q},\mathbbm{k}}|\leq C$,
\item sequences of functions $(f_{i,\mathcal{Q},\mathbbm{k}_i})_{\mathbbm{k}_i\in \Z^{d+1}}$ on $\mathbbm{G}$, for $i=1,2$ such that $\supp f_{i,\mathcal{Q},\mathbbm{k}_i}\subseteq \mathcal{Q}_i$ and $\|f_{i,\mathcal{Q},\mathbbm{k}_i}\|_{L^\infty}\leq 1$,
\end{enumerate} 
and this data relates to the kernel $k_\mathcal{Q}$ via the identity
$$k_\mathcal{Q}(x,y)=\sum_{\mathbbm{k}=(\mathbbm{k}_1,\mathbbm{k}_2)}\lambda_{\mathcal{Q},\mathbbm{k}} |\mathcal{Q}_1|^{-1/2} f_{1,\mathcal{Q},\mathbbm{k}_1}(x)|\mathcal{Q}_2|^{-1/2} f_{2,\mathcal{Q},\mathbbm{k}_2}(y).$$
\end{lem}

The decomposition of $k_\mathcal{Q}$ follows the idea in the proof of \cite[Lemma 2.9]{rochbergsemmes}: extend $k_\mathcal{Q}$ to a smooth compactly supported function $\tilde{k}_{\mathcal{Q}}$ on a $2\epsilon$-neighborhood of $\mathcal{Q}$ (for $\epsilon\sim\eta(\mathcal{Q}_i)/3$ so small that $\{x:\rd_{CC}(x,\mathcal{Q})<3\epsilon\}$ does not intersect the diagonal) by truncating $k$ using a smooth cutoff $\varphi_\mathcal{Q}$. The smooth cutoff $\varphi_\mathcal{Q}$ has to satisfy $\varphi_\mathcal{Q}\in C^\infty_c(\{x:\rd_{CC}(x,\mathcal{Q})<2\epsilon\})$, $\varphi_\mathcal{Q}|_{\mathcal{Q}}=1$ and that $|\partial^{\beta}\varphi_{\mathcal{Q}}|\lesssim \eta(\mathcal{Q}_1)^{-\langle \beta\rangle}$. As above, $\langle \beta\rangle:=2\beta_0+\sum _{i=1}^d \beta_i$ for $\beta=(\beta_0,\beta_1,\ldots,\beta_d)\in \N^{d+1}$. 

We expand $\tilde{k}_{\mathcal{Q}}$ in a double Fourier series on a large enough euclidean cube, 
$$\tilde{k}_{\mathcal{Q}}(x,y)=\sum_{\mathbbm{k}=(\mathbbm{k}_1,\mathbbm{k}_2)}\lambda_{\mathcal{Q},\mathbbm{k}} |\mathcal{Q}_1|^{-1/2} \pmb{e}_{\mathcal{Q},\mathbbm{k}_1}(x)|\mathcal{Q}_2|^{-1/2} \pmb{e}_{\mathcal{Q},\mathbbm{k}_2}(y),$$
where $\pmb{e}_{1,\mathcal{Q},\mathbbm{k}_1}$ and $\pmb{e}_{2,\mathcal{Q},\mathbbm{k}_2}$ are complex exponentials for $\mathbbm{k}_1,\mathbbm{k}_2\in\Z^2$. Since $\tilde{k}_\mathcal{Q}$ is smooth, standard estimates show that the coefficients $(\lambda_{\mathcal{Q},\mathbbm{k}})_{\mathbbm{k}\in I\times I}$ of $\tilde{k}_\mathcal{Q}$ in this expansion satisfy that for $a>0$ there is a constant $C>0$ (depending on $k$ and $a$) such that $|\mathbbm{k}|^a |\lambda_{\mathcal{Q},\mathbbm{k}}|\leq C$. We arrive at the conclusion of the lemma once setting 
$$f_{i,\mathcal{Q},\mathbbm{k}_i}=\mathbbm{e}_{\mathcal{Q},\mathbbm{k}_i}\chi_{\mathcal{Q}_i}.$$

Note that by Proposition \ref{lpboundsandhnwo}, the sequences $(|\mathcal{Q}_i|^{-1/2}f_{i,\mathcal{Q},\mathbbm{k}})_{\mathcal{Q}_i}$ form HNWO-sequences with uniformly bounded HNWO-constant. The argument to prove Theorem \ref{oscestimate} now proceeds as in \cite[Page 252-253]{rochbergsemmes}. We can write the integral kernel of $[Q,a]$ as 
\begin{align*}
\sum_{\mathcal{Q}\in \mathcal{P}}&\sum_{\mathbbm{k}\in \Z^{d+1}} \lambda_{\mathcal{Q},\mathbbm{k}} |\mathcal{Q}|^{-1/2} (a(x)-a_{\mathcal{Q}_1})f_{1,\mathcal{Q},\mathbbm{k}}(x)f_{2,\mathcal{Q},\mathbbm{k}}(y)\\
&+\sum_{\mathcal{Q}\in \mathcal{P}}\sum_{\mathbbm{k}\in \Z^{d+1}} \lambda_{\mathcal{Q},\mathbbm{k}} |\mathcal{Q}|^{-1/2} f_{1,\mathcal{Q},\mathbbm{k}}(x)(a_{\mathcal{Q}_1}-a(y)) f_{2,\mathcal{Q},\mathbbm{k}}(y).
\end{align*}
The HNWO-constants of the sequences appearing compare to the oscillation numbers of $a$ as in \cite[Page 252]{rochbergsemmes}, and the cubes in the Whitney decomposition $\mathcal{P}$ compare with the cubes in the dyadic decomposition $\mathcal{I}$ as in \cite[Page 253]{rochbergsemmes}.

\begin{remark}
In \cite{rochbergsemmes}, there are also sufficient conditions on the integral kernel of an operator $Q$ which guarantee that $[Q,a]\in \mathcal{L}^{p,q}(L^2(M))$ implies that  $a\in \mathrm{Osc}^{p,q}(M)$ (in the Riemannian setting). We note that since $\mathcal{L}^{p,q}$ is an operator ideal, the space $\{a:[Q,a]\in \mathcal{L}^{p,q}(L^2(M))\}$ forms a Banach $*$-algebra. Therefore, $\mathrm{Osc}^{p,q}(M)$ is a Banach $*$-algebra. As this line of questions is not relevant for the computation of Dixmier traces, we do not pursue it here.
\end{remark}

To apply Theorem \ref{oscestimate}, of course, the main issue is to identify the spaces $\mathrm{Osc}^{p,q}_H(M)$ and understand their topology. Based on the discussion in \cite[Chapter 4]{rochbergsemmes}, one is lead to suspect that $ \mathrm{Osc}^{p,q}_H(M)$ coincides with a sub-Riemannian Besov-type space if $p>d+2$ or $q<\infty$. However, it is unclear what the correct definition of Besov spaces is in the sub-Riemannian setting. We will be interested in the case $p=d+2$ and $q=\infty$. Here the results of \cite[Appendix]{connessulltele} lead one to believe that $\mathrm{Osc}^{d+2,\infty}_H(M)$ coincides with $W^{1,d+2}_H(M)$. We only require one of these inclusions.

\begin{prop}
\label{dinftyprop}
There is a bounded inclusion $W^{1,d+2}_H(M)\hookrightarrow \mathrm{Osc}^{d+2,\infty}_H(M)$.
\end{prop}

The horizontal Sobolev spaces were defined on page \pageref{horizontalestimaate} and the space $W^{1,p}_H(M)$ in Equation \eqref{horizontalestimaate}. The Riemannian analogue of Proposition \ref{dinftyprop} was proved in \cite[page 228]{otherrochbergsemmes}. 

\begin{proof}
By localizing and using a density argument, we may reduce the proof to showing the following; For $f\in W^{1,d+2}_H(\mathbbm{G})\cap C^1_c(\mathbbm{G})$, an estimate of the form $\|f\|_{\mathrm{Osc}^{d+2,\infty}}\leq C\|f\|_{W^{1,d+2}_H}$ holds. We define 
$$F(x,s):=f(x)$$ 
and estimate $F^*(x,s)\leq s|\nabla_Hf(x)|$, where $\nabla_H$ denotes the horizontal gradient on $\mathbbm{G}$. We compute that
\begin{align*}
m\{(x,s): &\;s|\nabla_Hf(x)|>u\}=\int_{s|\nabla_Hf(x)|>u}\int s^{-d-3}\rd s\rd x\\
&=\frac{1}{d+2}\int |\nabla_Hf(x)|^{d+2} u^{-d-2}\rd x=\frac{\||\nabla_Hf|\|_{L^{d+2}}}{d+2}u^{-d-2}\leq \frac{\|f\|_{W^{1,d+2}_H}}{d+2}u^{-d-2}.
\end{align*}
Therefore $\|F^*\|_{L^{p,q}(\mathbbm{G}\times \R_+, s^{-d-3}\rd s\rd x)}\leq C\|f\|_{W^{1,d+2}_H}$ and the proposition follows from the concluding remark of Lemma \ref{threeeq}.
\end{proof}

\begin{thm}
\label{extedi}
Let $M$ be an $n$-dimensional sub-Riemannian $H$-manifold and $T_0,T_1,\ldots, T_{n+1}\in \Psi^0_H(M)$. Consider the two $n+1$-linear mappings 
\begin{align}
\label{dtraaappcon}
C^\infty(M)^{\otimes n+1}\ni a_1\otimes \cdots \otimes a_{n+1}&\mapsto d(2\pi)^d\tra_\omega(T_0[T_1,a_1]\cdots [T_{n+1},a_{n+1}]),\\
\label{wresaappcon}
C^\infty(M)^{\otimes n+1}\ni a_1\otimes \cdots \otimes a_{n+1}&\mapsto \mathrm{WRes}_H(T_0[T_1,a_1]\cdots [T_{n+1},a_{n+1}]).
\end{align}
The two $n+1$-linear functionals in \eqref{dtraaappcon} and \eqref{wresaappcon} coincide and are both continuous in the $W^{1,n+1}_H$-topology. Moreover, for $a_1,\ldots, a_{n+1}\in W^{1,n+1}_H(M)$
$$\tra_\omega(T_0[T_1,a_1]\cdots [T_{n+1},a_{n+1}])= \mathrm{WRes}_H(T_0[T_1,a_1]\cdots [T_{n+1},a_{n+1}]).$$
In particular, $T_0[T_1,a_1]\cdots [T_{n+1},a_{n+1}]$ is measurable for $a_1,\ldots, a_{n+1}\in \Lip_{CC}(M)$.
\end{thm}

Here $\mathrm{WRes}_H$ denotes Ponge's Wodzicki residue in the Heisenberg calculus (see \cite{pongeresidue}). The proof proceeds exactly as the one  of Theorem \ref{ctfforlip} and is left to the reader. Using Theorem \ref{extedi}, we can extend a computation of Englis-Zhang \cite{engzh} to a larger class of functions. 

\begin{cor}
\label{engszhangcor}
Suppose that $M$ is an $n$-dimensional contact manifold and that $P_M\in \Psi^0_H(M)$ is a Szeg\"o projection. For any $a_1,\ldots, a_{n+1}\in W^{1,n+1}_H(M)$ we have the equality
\begin{align*}
\tra_\omega&\left([P_Ma_1P_M,P_Ma_2P_M]\cdots [P_Ma_{n}P_M,P_Ma_{n+1}P_M]\right)\\
&\qquad=\frac{1}{n!(2\pi)^n} \int_{\partial \Omega} \mathcal{L}^*(\bar{\partial}_ba_1,\bar{\partial}_ba_2)\cdots \mathcal{L}^*(\bar{\partial}_ba_{n},\bar{\partial}_ba_{n+1}) \rd V_\theta,
\end{align*}
where $\bar{\partial}_b$ denotes the boundary $\bar{\partial}$-operator, $\mathcal{L}^*$ the dual Levi form and $\rd V_\theta$ the volume form associated with the contact structure. 
\end{cor}

\section*{{\bf Acknowledgments}}

H. G. was supported by ERC Advanced Grant HARG 268105. M.G wishes to thank the Knut and Alice Wallenberg foundation for their support. M. G. was supported by the Swedish Research Council Grant 2015-00137 and Marie Sklodowska Curie Actions, Cofund, Project INCA 600398. We gratefully acknowledge additional travel support by the London Mathematical Society, the International Centre for Mathematical Sciences (Edinburgh) and the Centre for Symmetry and Deformation (Copenhagen). We thank Jon Johnsen, Jens Kaad, Steven Lord, Ryszard Nest, Fedor Sukochev and Dmitriy Zanin for fruitful discussions. We thank Michael E.~Taylor for pointing us to his unpublished preprint \cite{hbcom}. We are grateful to Richard Rochberg for suggesting the relevance of the results in \cite{rochbergsemmes} to contact manifolds.

\end{document}